\documentclass[11pt, twoside, leqno]{article}

\usepackage{mathrsfs}
\usepackage{amssymb}
\usepackage{amsmath}
\usepackage{mathrsfs}
\usepackage{amsthm}
\usepackage{amsfonts}
\usepackage{color}
\usepackage{latexsym}
\usepackage{txfonts}
\usepackage{indentfirst}
\usepackage{anysize}
\usepackage{tikz}

\allowdisplaybreaks

\newcounter{assum}

\newtheorem{theorem}{Theorem}[section]
\newtheorem{lemma}[theorem]{Lemma}
\newtheorem{corollary}[theorem]{Corollary}

\theoremstyle{definition}
\newtheorem{remark}[theorem]{Remark}

\newtheorem{definition}[theorem]{Definition}
\renewcommand{\appendix}{\par
	\setcounter{section}{0}%
	\setcounter{subsection}{0}%
	\setcounter{subsubsection}{0}%
	\gdef\thesection{\@Alph\c@section}%
	\gdef\thesubsection{\@Alph\c@section.\@arabic\c@subsection}%
	\gdef\theHsection{\@Alph\c@section.}%
	\gdef\theHsubsection{\@Alph\c@section.\@arabic\c@subsection}%
	\csname appendixmore\endcsname
}

\numberwithin{equation}{section}

\begin{document}

\title{\bf\Large Parabolic Extrapolation and Its Applications
to Characterizing Parabolic BMO Spaces via Parabolic Fractional Commutators
\footnotetext{\hspace{-0.35cm} 2020
{\it Mathematics Subject Classification}.
Primary 47B47; Secondary 42B25, 42B35, 47A30, 46E35.\endgraf
{\it Key words and phrases.}
parabolic extrapolation, PBMO, parabolic fractional integral
operator, parabolic commutator.
\endgraf
This project is partially supported by the National
Key Research and Development Program of China
(Grant No. 2020YFA0712900), the Spanish Ministry of
Science and Innovation through the Ram\'{o}n y Cajal 2021
(RYC2021-032600-I),
the National Natural Science Foundation of China
(Grant Nos. 12431006 and 12371093),
and the Fundamental Research Funds
for the Central Universities (Grant No. 2233300008).}}
\date{}
\author{Mingming Cao, Weiyi Kong, Dachun
Yang, Wen Yuan\footnote{Corresponding author,
E-mail: \texttt{wenyuan@bnu.edu.cn}/{\color{red}
\today}/Final version.} \ and Chenfeng Zhu}

\maketitle

\vspace{-0.7cm}

\begin{center}
\begin{minipage}{13cm}
{\small {\bf Abstract}\quad In this article, we establish the
parabolic version of the celebrated Rubio de Francia
extrapolation theorem. As applications, we obtain new
characterizations of parabolic BMO-type spaces in terms of
various commutators of parabolic fractional operators with time
lag. The key tools to achieve these include to establish
the appropriate form in the parabolic setting of the
parabolic Rubio de Francia iteration algorithm, the Cauchy
integral trick, and a modified Fourier series expansion argument
adapted to the parabolic geometry. The novelty of these results
lies in the fact that, for the first time, we not only introduce
a new class of commutators associated with parabolic fractional
integral operators with time lag, but also utilize them to
provide a characterization of the parabolic BMO-type space in the high-dimensional case.
}
\end{minipage}
\end{center}

\vspace{0.2cm}

\tableofcontents

\vspace{0.2cm}

\section{Introduction}

The classical Muckenhoupt weights, introduced in
\cite{m(tams-1972)} in the early 1970s, play a crucial role in
harmonic analysis and partial differential equations and have
been thoroughly developed in the past few decades; see, for
example, \cite{Cruz-book, Rubio-book}. The
parabolic Muckenhoupt weights with time lag were originally
introduced by Kinnunen and Saari
\cite{ks(na-2016), ks(apde-2016)}. The motivation for studying
the theory of parabolic Muckenhoupt weights with time lag stems
from the following two aspects. On the one hand, it has a very
profound connection with the regularity theory of solutions
of the doubly nonlinear parabolic equation
\begin{align}\label{20250116.2142}
\frac{\partial}{\partial t}\left(|u|^{p-2}u\right)-\mathrm{div}
\left(\left|\nabla u\right|^{p-2}\nabla u\right)=0.
\end{align}
Here, and thereafter, we \emph{always fix} $p\in(1,\infty)$.
On the other hand, the parabolic Muckenhoupt condition proves
an appropriate substitute in the higher-dimensional case of
the following one-sided Muckenhoupt
condition in the one-dimensional case
\begin{align*}
[\omega]_{A_q^+(\mathbb{R})}:=
\sup_{\genfrac{}{}{0pt}{}{x\in\mathbb{R}}{h\in(0,\infty)}}\frac{1}{h}
\int_{x-h}^x\omega(y)\,dy\left\{\frac{1}{h}
\int_x^{x+h}\left[\omega(y)\right]
^{\frac{1}{1-q}}\,dy\right\}^{q-1}<\infty
\end{align*}
introduced by Sawyer \cite{s(tams-1986)} in connection with
ergodic theory, where $q\in(1,\infty)$ and $\omega$ is a
locally integrable function on $\mathbb{R}$ which is positive
almost everywhere. We refer to \cite{kk(ma-2007), m(cpam-1961),
m(cpam-1964), t(cpam-1968)} for more investigations into the
equation \eqref{20250116.2142}, to \cite{a(tams-1988),
fg(pams-1985), kmy(ma-2023), kyyz-2024, my(mz-2024),
s(rmi-2016), s(ampa-2018)} for more studies on parabolic
function spaces related to \eqref{20250116.2142}, and to
\cite{afm(pams-1997), cno(sm-1993), fmo(tams-2011),
flps(na-2014), mpd(cjm-1993), md(jlms-1994), o(pams-2005)} for
further research about the one-sided Muckenhoupt weights.
The concept of parabolic Muckenhoupt weights with time lag is
built upon the following parabolic rectangles,
which are derived from the intrinsic geometry of
\eqref{20250116.2142} and the
parabolic Harnack's inequality associated with
\eqref{20250116.2142}. For any $x\in\mathbb{R}^n$ and
$L\in(0,\infty)$, let $Q(x,L)$ be the \emph{cube} in
$\mathbb{R}^n$ centered at $x$ with edge length $2L$
and all its edges parallel to the coordinate axes.

\begin{definition}\label{parabolic rectangles}
Let $(x,t)\in\mathbb{R}^{n+1}$ and $L\in(0,\infty)$. A
\emph{parabolic rectangle} $R(x,t,L)$ centered at
$(x,t)$ with edge length $L$ is defined by setting
$R:=R(x,t,L):=Q(x,L)\times(t-L^{p},t+L^{p}).$
The \emph{$\gamma$-upper part}
$R^+(\gamma)$ and the \emph{$\gamma$-lower part}
$R^-(\gamma)$ of $R$ are defined, respectively, by setting
\begin{align*}
R^+(\gamma):=Q(x,L)\times(t+\gamma L^p,t+L^{p})
\end{align*}
and
\begin{align*}
R^-(\gamma):=Q(x,L)\times(t-L^{p},t-\gamma L^p),
\end{align*}
where $\gamma\in[0,1)$ is called the \emph{time lag}.
\end{definition}

We simply denote $R(x,t,L)$ by $R$ and let
$\mathcal{R}_p^{n+1}$ be the set of all parabolic rectangles
in $\mathbb{R}^{n+1}$. In what follows, if there exists no
confusion, then we \emph{always omit} the differential
$dx\,dt$ in all integral representations to simplify the
presentation and we \emph{always suppress} the variables
$(x,t)$ in the notation and, for instance, for any
$A\subset\mathbb{R}^{n+1}$, any function $f$ on
$\mathbb{R}^{n+1}$, and any $\lambda\in\mathbb{R}$, we
simply write $A\cap\{f>\lambda\}:=\{(x,t)\in A:\
f(x,t)>\lambda\}$. Let $L_{\mathrm{loc}}^1$ denote the set of all
locally integrable functions on $\mathbb{R}^{n+1}$.
Recall that a locally integrable function on $\mathbb{R}^{n+1}$ which is positive
almost everywhere is called a \emph{weight}. For any $f\in
L_{\mathrm{loc}}^1$ and any measurable set
$A\subset\mathbb{R}^{n+1}$ with $\vert A\vert\in(0,\infty)$, let
\begin{align*}
f_A:=\fint_Af:=\frac{1}{\left\vert A\right\vert}\int_Af.
\end{align*}
We then have the following parabolic Muckenhoupt
two-weight classes with time lag from \cite[p.\,3]{cm24}
and \cite[Definition 2.1]{kyyz-2024-2}.

\begin{definition}\label{parabolic weight}
Let $\gamma\in[0,1)$ and
$1\leq r\leq q\leq\infty$. The \emph{parabolic
Muckenhoupt two-weight class $TA_{r,q}^+(\gamma)$ with time
lag} is defined to be the set of all pairs $(u,v)$ of weights
on $\mathbb{R}^{n+1}$ such that
\begin{align*}
[u,v]_{TA_{r,q}^+(\gamma)}:=\sup_{R\in\mathcal{R}_p^{n+1}}
\left[\fint_{R^-(\gamma)}u^q\right]^{\frac{1}{q}}
\left[\fint_{R^+(\gamma)}v^{-r'}\right]^{\frac{1}{r'}}
<\infty
\end{align*}
with usual modifications made when $r=1$ or $q=\infty$.
If this inequality holds with the time axis reversed,
then $(u,v)$ is said to belong to $TA_{r,q}^-(\gamma)$ which is also called the
\emph{parabolic Muckenhoupt two-weight class with time lag}.
\end{definition}

\begin{remark}\label{parabolic weight remark}
Let $\gamma\in[0,1)$ and $1\leq r\leq q<\infty$. We denote by
$A_{r,q}^+(\gamma)$ the set of all weights $\omega$ such
that $(\omega,\omega)\in TA_{r,q}^+(\gamma)$ and let
$[\omega]_{A_{r,q}^+(\gamma)}:=[\omega,\omega]
_{TA_{r,q}^+(\gamma)}$, and denote by $A_q^+(\gamma)$ the set of all weights $\omega$ such
that $\omega^\frac{1}{q}\in A_{q,q}^+(\gamma)$ and let
$[\omega]_{A_q^+(\gamma)}:=[\omega^\frac{1}{q}]
_{A_{q,q}^+(\gamma)}^q$. The weight classes
$A_{r,q}^-(\gamma)$ and $A_q^-(\gamma)$ are defined
similarly.
\end{remark}

In the past few years, the parabolic Muckenhoupt weights with
time lag have achieved significant advancements and
attracted a lot of attention; see, for example, \cite{cm24,
km(am-2024), 2310.00370, kmyz(pa-2023), kyyz-2024-2,
mhy(fm-2023)}.

One of the most powerful and valuable tools in the theory
of weighted norm inequalities is the famous extrapolation theorem
of Rubio de Francia \cite{r(ajm-1984)}. For example, it has
been applied to the proof of the well-known $A_2$ conjecture
(see \cite{dgpp(pm-2005), h(am-2012), h(cm-2014), l(jam-2013)}),
and it can also be used to deduce the
vector-valued weighted norm inequalities from the
scalar-valued ones. We refer to \cite{cmm(am-2022),
co(mn-2024), cm(ma-2018), Cruz-book, hms(ajm-1988),
lmo(am-2020)} for developments and various
generalizations of extrapolation and to \cite{clm(jam-2018),
lr(mia-2007), mr(prse-2000)} for the one-sided case.

Recall that, for any $q\in(0,\infty)$ and any weight $\omega$,
the \emph{weighted Lebesgue space $L^q(\omega)$}
is defined to be the set of all measurable functions $f$
on $\mathbb{R}^{n+1}$ such that
\begin{align*}
\|f\|_{L^q(\omega)}:=\left(\int_{\mathbb{R}^{n+1}}
|f|^q\omega\right)^\frac{1}{q}<\infty
\end{align*}
and the \emph{weighted weak Lebesgue space
$L^{q,\infty}(\omega)$} is defined to
be the set of all measurable functions $f$ on
$\mathbb{R}^{n+1}$ such that
\begin{align*}
\|f\|_{L^{q,\infty}(\omega)}:=
\sup_{\lambda\in(0,\infty)}\lambda
\left[\omega(\{|f|>\lambda\})\right]^\frac{1}{q}<\infty.
\end{align*}
In particular, when $\omega\equiv 1$, $L^q(\omega)$ reduces to
the \emph{Lebesgue space} $L^q$ and $L^{q,\infty}(\omega)$
to the \emph{weak Lebesgue space} $L^{q,\infty}$. Also recall that the \emph{Lebesgue space
$L^\infty$} is defined to be the set of all measurable functions
$f$ on $\mathbb{R}^{n+1}$ such that
$\|f\|_{L^\infty}:=\mathrm{ess\,sup}_{(x,t)
\in\mathbb{R}^{n+1}}|f(x,t)|<\infty$.

In this article, we first establish the following parabolic version of Rubio
de Francia extrapolation theorem (see also Corollaries
\ref{parabolic extrapotion cor 1}-\ref{parabolic extrapotion cor 3}
for its extensions, Theorem
\ref{parabolic extrapolation infinity} for the extrapolation
at infinity, and Theorem
\ref{parabolic extrapolation for commutator} for the
extrapolation of commutators of linear operators).

\begin{theorem}\label{parabolic extrapolation}
Let $\gamma\in(0,1)$ and $1\leq r_0\leq q_0<\infty$. Let
$\mathcal{F}$ be the set of all pairs $(f,g)$ of
measurable functions on $\mathbb{R}^{n+1}$ such that,
for any $\omega_0\in A_{r_0,q_0}^+(\gamma)$,
\begin{align}\label{20241204.1150}
\|f\|_{L^{q_0}(\omega_0^{q_0})}\leq C_0
\|g\|_{L^{r_0}(\omega_0^{r_0})}
\end{align}
holds provided that the left-hand side of
\eqref{20241204.1150} is finite, where the positive constant
$C_0$ in \eqref{20241204.1150} depends only on $\mathcal{F}$,
$n$, $p$, $\gamma$, $r_0$, $q_0$, and
$[\omega_0]_{A_{r_0,q_0}^+(\gamma)}$. Then, for any
$1<r\leq q<\infty$ satisfying
\begin{align}\label{1113}
\frac{1}{r}-\frac{1}{q}=\frac{1}{r_0}-\frac{1}{q_0}
\end{align}
and for any $\omega\in A_{r,q}^+(\gamma)$, there
exists a positive constant $C$, depending only on
$\mathcal{F}$, $n$, $p$, $\gamma$, $r_0$, $q_0$, $r$, $q$, and
$[\omega]_{A_{r,q}^+(\gamma)}$, such that, for any
$(f,g)\in\mathcal{F}$,
\begin{align}\label{20241204.1152}
\|f\|_{L^q(\omega^q)}\leq C
\|g\|_{L^r(\omega^r)}
\end{align}
holds provided that the left-hand side of
\eqref{20241204.1152} is finite.
\end{theorem}

\begin{remark}
Theorem \ref{parabolic extrapolation} with $r_0=q_0$ is the
parabolic version of the classical Rubio de Francia
extrapolation theorem, while Theorem \ref{parabolic extrapolation}
with $r_0<q_0$ is the parabolic version
of the off-diagonal extrapolation theorem due to Harboure,
Mac\'ias, and Segovia \cite{hms(ajm-1988)}.
\end{remark}

Let $\beta\in\mathbb{R}$. Recall that the \emph{parabolic Campanato space
$\mathrm{PC}^\beta$} is defined to be the space of
all $f\in L_{\mathrm{loc}}^1$ such that
\begin{align*}
\|f\|_{\mathrm{PC}^\beta}:=
\sup_{R\in\mathcal{R}_p^{n+1}}\frac{1}{|R|^\beta}
\fint_R\left|f-f_R\right|<\infty.
\end{align*}
The space $\mathrm{PC}^0$ is precisely the space of functions
with bounded mean oscillation defined
via parabolic rectangles, which is usually denoted by PBMO.
Also recall that the \emph{parabolic
distance} $d_p$ on $\mathbb{R}^{n+1}\times\mathbb{R}^{n+1}$ is
defined by setting, for any $(x,t),(y,s)\in\mathbb{R}^{n+1}$,
\begin{align*}
d_p((x,t),(y,s)):=\max\left\{\|x-y\|_\infty,\,
|t-s|^{\frac{1}{p}}\right\}.
\end{align*}
Via this parabolic distance, we now introduce the $k$th order commutators of
parabolic fractional integral operators with time lag as follows.
\begin{definition}
Let $\gamma,\alpha\in[0,1)$, $k\in\mathbb{N}$, and $b\in
L_{\mathrm{loc}}^1$. The \emph{$k$th order commutator
$[b,I_\alpha^{\gamma+}]_k$ of parabolic fractional
integral operators with time lag} is defined by setting, for any
$f\in L_{\mathrm{loc}}^1$ and $(x,t)\in\mathbb{R}^{n+1}$,
\begin{align*}
\left[b,I_\alpha^{\gamma+}\right]_k(f)(x,t):=
\int_{\bigcup_{L\in(0,\infty)}R
(x,t,L)^+(\gamma)}\frac{[b(x,t)-b(y,s)]^kf(y,s)}
{[d_p((x,t),(y,s))]^{(n+p)(1-\alpha)}}\,dy\,ds.
\end{align*}
\end{definition}

The novelty of these commutators lies in their domains and,
moreover, in \cite{kyyz-2024-2}, the weighted boundedness of the corresponding
parabolic fractional integral operators with time lag
was used to characterize $A_{r,q}^+(\gamma)$
weights (see also Lemma \ref{weighted inequality fractional
integral}).

As an application of the parabolic extrapolation theorem, we
obtain the following new characterization
of the parabolic BMO space in terms of $k$th order commutators of
parabolic fractional integral operators with time lag.

\begin{theorem}\label{[I_alpha,b]}
Let $\gamma,\alpha\in(0,1)$ and $b\in L_{\mathrm{loc}}^1$.
Then the following statements are mutually equivalent.
\begin{enumerate}
\item[\rm(i)] $b\in\mathrm{PBMO}$.

\item[\rm(ii)] For any $k\in\mathbb{N}$, $1<r<q<\infty$
satisfying $\frac{1}{r}-\frac{1}{q}=\alpha$, and $\omega\in
A_{r,q}^+(\gamma)$, $[b,I_\alpha^{\gamma+}]_k$ is bounded from
$L^r(\omega^r)$ to $L^q(\omega^q)$.

\item[\rm(iii)] For some $1<r<q<\infty$ satisfying
$\frac{1}{r}-\frac{1}{q}=\alpha$, $[b,I_\alpha^{\gamma+}]$ is
bounded from $L^r$ to $L^q$.
\end{enumerate}
\end{theorem}

\begin{remark}
\begin{enumerate}
\item[\rm(i)] We prove Theorem \ref{[I_alpha,b]} by following
the chain of reasoning:
$\rm(i)\ \Longrightarrow\ \rm(ii)\Longrightarrow\ \rm(iii)\
\Longrightarrow\ \rm(i).$
It is important to note that, in the process of deriving (iii)
from (i), we make use of the parabolic extrapolation theorem (namely
Theorem \ref{parabolic extrapolation for commutator}) for the
commutator of the parabolic fractional integral operator
$I_\alpha^{\gamma+}$ with time lag [see \eqref{20250304.2134}
for the definition of $I_\alpha^{\gamma+}$]. Thus, in some sense,
the parabolic weights with time lag are necessary to obtain the equivalence
between (i) and (iii) of Theorem \ref{[I_alpha,b]}.

\item[\rm(ii)] To the best of our knowledge, Theorem
\ref{[I_alpha,b]} is the first result to characterize the
parabolic BMO space via the commutator of parabolic fractional
integral operators with time lag. For related one-sided results,
see, for instance, \cite[Corollary 12]{lmr(jmaa-2024)} and
\cite[Theorem 2.9]{lr(prse-2005)}, which show that,
in the one-dimensional case,
if $b$ belongs to $\mathrm{BMO}$, then, for any $k\in\mathbb{N}$,
the $k$th order commutator of the one-sided fractional integral
operator is one-sided weighted bounded. Different from
\cite[Corollary 12]{lmr(jmaa-2024)} and
\cite[Theorem 2.9]{lr(prse-2005)}, which only provide the
sufficient condition for the weighted boundedness
of $k$th order commutators in the one-dimensional case, the
advantages of Theorem \ref{[I_alpha,b]} are not only in the
higher-dimensional case, but also give the sufficient and
necessary condition for the weighted boundedness of $k$th order
commutators.
\end{enumerate}
\end{remark}

We also obtain the following various new characterizations of parabolic Campanato
spaces in terms of various $k$th parabolic fractional
commutators. We refer to Definitions
\ref{parabolic maximal commutator} and
\ref{maximal commutators of fractional integral}, respectively,
for the concepts of
$M_{\alpha,b}^{\gamma+,k}$ and $I_{\alpha,b}^{\gamma+,k}$.

\begin{theorem}\label{commutator theorem 2}
Let $\gamma,\sigma\in(0,1)$, $\rho,\alpha\in[0,1)$,
$\beta\in(0,\frac{1}{n+p})$, and $b\in
L_{\mathrm{loc}}^1$. Then the following
assertions are mutually equivalent.
\begin{enumerate}
\item[\rm(i)] $b\in\mathrm{PC}^\beta$.

\item[\rm(ii)] For any $k\in\mathbb{N}$, $1<r<q<\infty$
satisfying $\frac{1}{r}-\frac{1}{q}=\alpha+\beta k$, and
$\omega\in A_{r,q}^+(\gamma)$, $M_{\alpha,b}^{\gamma+,k}$ is
bounded from $L^r(\omega^r)$ to $L^q(\omega^q)$.

\item[\rm(iii)] For some $1<r<q<\infty$ satisfying
$\frac{1}{r}-\frac{1}{q}=\alpha+\beta$,
$M_{\alpha,b}^{\gamma+,1}$ is bounded from $L^r$ to $L^q$.

\item[\rm(iv)] For any $k\in\mathbb{N}$, $1\leq r<q<\infty$
satisfying $\frac{1}{r}-\frac{1}{q}=\alpha+\beta k$, and
$\omega\in A_{r,q}^+(\rho)$, $M_{\alpha,b}^{\rho+,k}$ is bounded
from $L^r(\omega^r)$ to $L^{q,\infty}(\omega^q)$.

\item[\rm(v)] For some $1\leq r<q<\infty$ satisfying
$\frac{1}{r}-\frac{1}{q}=\alpha+\beta$, $M_{\alpha,b}^{\rho+,1}$
is bounded from $L^r$ to $L^{q,\infty}$.

\item[\rm(vi)] For any $k\in\mathbb{N}$, $r\in(1,\infty)$
satisfying $\frac{1}{r}=\alpha+\beta k$, $(u,v)\in
TA_{r,\infty}^+(\rho)$, and $f\in L^r(\omega^r)$,
\begin{align*}
\left\|M_{\alpha,b}^{\rho+,k}(f)u
\right\|_{L^{\infty}}\lesssim
\|fv\|_{L^r},
\end{align*}
where the implicit positive constant is independent of $f$.

\item[\rm(vii)] $M_{\alpha,b}^{\rho+,1}$
is bounded from $L^r$ to $L^\infty$, where $r\in(1,\infty)$
satisfies $\frac{1}{r}=\alpha+\beta$.

\item[\rm(viii)] For any $k\in\mathbb{N}$, $1<r<q<\infty$
satisfying $\frac{1}{r}-\frac{1}{q}=\sigma+\beta k$, and
$\omega\in A_{r,q}^+(\gamma)$, $[b,I_\sigma^{\gamma+}]_k$ is
bounded from $L^r(\omega^r)$ to $L^q(\omega^q)$.

\item[\rm(ix)] For some $1<r<q<\infty$ satisfying
$\frac{1}{r}-\frac{1}{q}=\sigma+\beta$, $[b,I_\sigma^{\gamma+}]$
is bounded from $L^r$ to $L^q$.

\item[\rm(x)] For any $k\in\mathbb{N}$, $1<r<q<\infty$
satisfying $\frac{1}{r}-\frac{1}{q}=\sigma+\beta k$, and
$\omega\in A_{r,q}^+(\gamma)$, $I_{\sigma,b}^{\gamma+,k}$ is
bounded from $L^r(\omega^r)$ to $L^q(\omega^q)$.

\item[\rm(xi)] For some $1<r<q<\infty$ satisfying
$\frac{1}{r}-\frac{1}{q}=\sigma+\beta$,
$I_{\sigma,b}^{\gamma+,1}$ is bounded from $L^r$ to $L^q$.

\item[\rm(xii)] For any $k\in\mathbb{N}$, $1\leq r<q<\infty$
satisfying $\frac{1}{r}-\frac{1}{q}=\sigma+\beta k$, and
$\omega\in A_{r,q}^+(\gamma)$, $I_{\sigma,b}^{\gamma+,k}$ is
bounded from $L^r(\omega^r)$ to $L^{q,\infty}(\omega^q)$.

\item[\rm(xiii)] For some $1\leq r<q<\infty$ satisfying
$\frac{1}{r}-\frac{1}{q}=\sigma+\beta$,
$I_{\sigma,b}^{\gamma+,1}$ is bounded from $L^r$ to
$L^{q,\infty}$.
\end{enumerate}
All the operator norms in (ii) through (xiii) are equivalent to
$\|b\|_{\mathrm{PC}^\beta}^k$.
\end{theorem}

\begin{remark}
We show Theorem \ref{[I_alpha,b]} by the following chains:
\begin{align*}
&\rm(i)\ \Longrightarrow\ \rm(ii)\ \Longrightarrow\ \rm(iii)\
\Longrightarrow\ \rm(i),\
\rm(i)\ \Longrightarrow\ \rm(iv)\ \Longrightarrow\ \rm(v)\
\Longrightarrow\ \rm(i),\\
&\rm(i)\ \Longrightarrow\ \rm(vi)\ \Longrightarrow\ \rm(vii)\
\Longrightarrow\ (i),\
\rm(i)\ \Longrightarrow\ \rm(viii)\ \Longrightarrow\ \rm(ix)\
\Longrightarrow\ \rm(i),\\
&\rm(i)\ \Longrightarrow\ \rm(x)\ \Longrightarrow\ \rm(xi)\
\Longrightarrow\ \rm(iii),\ \mathrm{and}\\
&\rm(i)\ \Longrightarrow\ \rm(xii)\ \Longrightarrow\ \rm(xiii)\
\Longrightarrow\ \rm(v).
\end{align*}
\end{remark}

The organization of the remainder of this article is as follows.

In Section \ref{section2}, we first construct the parabolic
Rubio de Francia iteration algorithm to prove
Theorem \ref{parabolic extrapolation}. After that, we present
its several extensions, including the weak type
parabolic extrapolation theorem (see Corollary
\ref{parabolic extrapotion cor 1}), the parabolic
$A_\infty$ extrapolation theorem (see Corollary
\ref{parabolic extrapotion cor 4}), and the vector-valued
version of the parabolic extrapolation theorem (see Corollary
\ref{parabolic extrapotion cor 2}). Then we establish the
parabolic extrapolation theorem at infinity (see Theorem
\ref{parabolic extrapolation infinity}) by borrowing some
ideas from \cite{mr(prse-2000)}. Finally, combining
Theorem \ref{parabolic extrapolation} and the so-called Cauchy
integral trick in \cite{abkp(sm-1993)} adapted to the parabolic
setting, we show the parabolic extrapolation theorem of
commutators (see Theorem
\ref{parabolic extrapolation for commutator}).

In Section \ref{section3}, we give several
characterizations of both the parabolic BMO space and parabolic
Campanato spaces. In Subsection \ref{subsection3.1}, we apply the
coincidence of the parabolic Campanato space and the parabolic
Lipschitz space to prove the mutual equivalences among (i) through (vii) of Theorem
\ref{commutator theorem 2}. Based on
this, in Subsection \ref{subsection3.2}, applying a classical
lemma about commutators of positive quasilinear operators in
\cite{bmr(pams-2000)} (namely Lemma
\ref{commutator for positive quasilinear operators}), we
provide characterizations of both the parabolic BMO space and
parabolic Campanato spaces via the boundedness on $L^q$ of
commutators of parabolic maximal operators with time lag (see
Theorem \ref{commutator theorem 1}). In Subsection
\ref{subsection3.3}, combining Theorem
\ref{parabolic extrapolation for commutator}, the Fourier series
expansion argument, and the parabolic geometry associated with
the integral domain of the parabolic fractional integral
operator with time lag, we show Theorem \ref{[I_alpha,b]}.
Finally, we prove the mutual equivalences among (i) and (viii)
through (xiii) of Theorem \ref{commutator theorem 2}.

At the end of this introduction, we make some conventions on
notation. Throughout this article, let
$\mathbb{N}:=\{1,2,\ldots\}$ and $\mathbb{Z}_{+}:=\mathbb{N}
\cup\{0\}$. For any $r\in[1,\infty]$, let $r'$ be the conjugate
number of $r$, that is, $\frac1r+\frac{1}{r'}=1$. For
any $x:=(x_1,\dots,x_n)\in\mathbb{R}^n$,
let $\|x\|_\infty:=\max\{|x_1|,\dots,|x_n|\}$ and
$|x|:=\sqrt{|x_1|^2+\dots+|x_n|^2}$. Let
$\mathbf{0}$ denote the origin of $\mathbb{R}^{n}$.
For any $A\subset\mathbb{R}^{n+1}$ and $a\in\mathbb{R}$,
let $A+(\mathbf{0},a):=\{(x,t+a):\ (x,t)\in A\}$ and
$A-(\mathbf{0},a):=\{(x,t-a):\ (x,t)\in A\}$.
For any measurable set $A\subset\mathbb{R}^{n+1}$, we denote
by $\vert A\vert$ its $(n+1)$-dimensional Lebesgue measure.
For any weight $\omega$ and any measurable set
$A\subset\mathbb{R}^{n+1}$, let
$\omega(A):=\int_A\omega$. For any given Banach spaces
$\mathscr{X}$ and $\mathscr{Y}$ and for any bounded operator
$T:\ \mathscr{X}\to\mathscr{Y}$, we denote the operator norm of
$T$ from $\mathscr{X}$ to $\mathscr{Y}$ by
$\|T\|_{\mathscr{L}(\mathscr{X},\mathscr{Y})}$. In particular,
if $\mathscr{X}=\mathscr{Y}$, then we simply write
$\|T\|_{\mathscr{L}(\mathscr{X},\mathscr{Y})}$ as
$\|T\|_{\mathscr{L}(\mathscr{X})}$. The symbol $C$ denotes a
positive constant which is independent
of the main parameters involved, but may vary from line to line.
The symbol $f\lesssim g$ means that
there exists a positive constant $C$ such that $f\le Cg$ and, if
$f\lesssim g\lesssim f$, we then write $f\sim g$. Finally, when
we show a theorem (and the like), in its proof we always use
the same symbols as those appearing in the statement itself of
that theorem (and the like).

\section{Parabolic Extrapolation}
\label{section2}

This section consists of three subsections. In Subsection
\ref{subsection2.1}, we first prove Theorem
\ref{parabolic extrapolation}. After that, we give its several
extensions, while, in Subsection \ref{subsection2.2},
we establish the parabolic extrapolation at
infinity, namely Theorem \ref{parabolic extrapolation infinity}.
Finally, in Subsection \ref{subsection2.3},
as an application, we provide a parabolic extrapolation
theorem for commutators of linear operators.

\subsection{Parabolic Rubio de Francia Extrapolation and Its Extensions}
\label{subsection2.1}

We begin with recalling the concept of uncentered parabolic
fractional maximal operators with time lag; see also
\cite[Definition 2.4]{kyyz-2024-2}.

\begin{definition}\label{parabolic maximal operator}
Let $\gamma,\alpha\in[0,1)$. The \emph{uncentered
parabolic fractional maximal operators $M^{\gamma+}_\alpha$
and $M^{\gamma-}_\alpha$ with time lag} are defined,
respectively, by setting, for any $f\in L_{\mathrm{loc}}^1$ and
$(x,t)\in\mathbb{R}^{n+1}$,
\begin{align*}
M^{\gamma+}_\alpha(f)(x,t):=\sup_{\genfrac{}{}{0pt}{}{R\in
\mathcal{R}_{p}^{n+1}}{(x,t)\in R^-(\gamma)}}
\left|R^+(\gamma)\right|^\alpha
\fint_{R^+(\gamma)}|f|
\end{align*}
and
\begin{align*}
M^{\gamma-}_\alpha(f)(x,t):=\sup_{\genfrac{}{}{0pt}{}{R\in
\mathcal{R}_{p}^{n+1}}{(x,t)\in R^+(\gamma)}}
\left|R^-(\gamma)\right|^\alpha
\fint_{R^-(\gamma)}|f|.
\end{align*}
\end{definition}

We simply denote $M^{\gamma+}_0$ and $M^{\gamma-}_0$,
respectively, by $M^{\gamma+}$ and $M^{\gamma-}$, which are
also called the \emph{uncentered
parabolic maximal operators}; see \cite[Definition
2.2]{km(am-2024)}. To show Theorem
\ref{parabolic extrapolation}, we need the following lemma.

\begin{lemma}\label{weighted inequality uncentered}
Let $\gamma\in(0,1)$ and $\alpha,\rho\in[0,1)$, and let
$\omega$, $u$, and $v$ be weights.
\begin{enumerate}
\item[\rm(i)] Let $1<r\leq q<\infty$ with
$\frac{1}{r}-\frac{1}{q}=\alpha$. Then $\omega\in
A_{r,q}^+(\gamma)$ if and only if $M^{\gamma+}_\alpha$ is
bounded from $L^r(\omega^r)$ to
$L^q(\omega^q)$.

\item[\rm(ii)] Let $1\leq r\leq q<\infty$
with $\frac{1}{r}-\frac{1}{q}=\alpha$. Then $\omega\in
A_{r,q}^+(\rho)$ if and only if $M^{\rho+}_\alpha$ is
bounded from $L^r(\omega^r)$ to
$L^{q,\infty}(\omega^q)$.

\item[\rm(iii)] Let $r\in(1,\infty]$ satisfy $\frac{1}{r}=
\alpha$. Then $(u,v)\in TA_{r,\infty}^+(\rho)$ if and only
if there exists a positive constant $C$ such that, for any
$f\in L^r(v^r)$,
\begin{align}\label{20241112.2251}
\left\|M^{\rho+}_\alpha(f)u\right\|_{L^\infty}\leq C\|fv\|_{L^r}.
\end{align}
\end{enumerate}
\end{lemma}

\begin{proof}
(i) is given in \cite[Corollary 4.4]{kyyz-2024-2}, while (ii)
is given in \cite[Theorems 4.1 and 4.3]{kyyz-2024-2}.
Next, we prove the sufficiency of (iii). Fix
$R\in\mathcal{R}_p^{n+1}$. Assume that \eqref{20241112.2251}
holds. From \eqref{20241112.2251} with $f$ therein
replaced by $v^{-r'}\mathbf{1}_{R^+(\rho)}$ and from
Definition \ref{parabolic maximal operator},
we infer that, for almost every $(x,t)\in R^-(\rho)$,
\begin{align*}
u(x,t)\left|R^+(\rho)\right|^\alpha\fint_{R^+(\rho)}v^{-r'}
&\leq u(x,t)M^{\rho+}_\alpha\left(v^{-r'}
\boldsymbol{1}_{R^+(\rho)}\right)(x,t)\\
&\leq\left\|uM^{\rho+}_\alpha
\left(v^{-r'}\boldsymbol{1}_{R^+(\rho)}\right)\right\|
_{L^\infty}\lesssim
\left\|v^{1-r'}\boldsymbol{1}_{R^+(\rho)}
\right\|_{L^r}\\
&=\begin{cases}
\left[\int_{R^+(\rho)}v^{-r'}\right]^{\frac{1}{r}}&\mbox{if\ }
r\in(1,\infty),\\
1&\mbox{if\ }r=\infty,
\end{cases}
\end{align*}
which, together with $\frac{1}{r}=\alpha$, further implies that
\begin{align*}
\mathop\mathrm{ess\,sup}_{(x,t)\in R^-(\rho)}u(x,t)
\lesssim\left[\fint_{R^+(\rho)}v^{-r'}\right]^{-\frac{1}{r'}}.
\end{align*}
Taking the supremum over all $R\in\mathcal{R}_p^{n+1}$,
we then conclude that $(u,v)\in TA_{r,\infty}^+(\rho)$.
This finishes the proof of the sufficiency of (iii).

Now, we show the necessity of (iii). Let $f\in
L_{\mathrm{loc}}^1$ be such that $fv\in
L^r$. From H\"older's inequality and $\frac{1}{r}=\alpha$, it
follows that, for any $R\in\mathcal{R}_p^{n+1}$ and almost every
$(x,t)\in R^-(\gamma)$,
\begin{align*}
u(x,t)\left|R^+(\rho)\right|^\alpha\fint_{R^+(\rho)}|f|
&\leq u(x,t)\left|R^+(\rho)\right|^{\alpha-1}
\left\|fv\boldsymbol{1}_{R^+(\rho)}\right\|
_{L^r}\left[\int_{R^+(\rho)}v^{-r'}\right]
^{\frac{1}{r'}}\\
&\leq\mathop\mathrm{ess\,sup}_{(y,s)\in
R^-(\rho)}u(y,s)
\left[\fint_{R^+(\rho)}v^{-r'}\right]^{\frac{1}{r'}}\|fv\|_{L^r}\\
&\leq[u,v]_{TA_{r,\infty}^+(\rho)}\|fv\|_{L^r},
\end{align*}
which further implies that \eqref{20241112.2251} holds. This
finishes the proof of (iii) and hence
Lemma \ref{weighted inequality uncentered}.
\end{proof}

By borrowing some ideas from the proof of \cite[Theorem
3.23]{Cruz-book}, we are ready to
prove Theorem~\ref{parabolic extrapolation}.

\begin{proof}[Proof of Theorem \ref{parabolic extrapolation}]
Let $1<r\leq q<\infty$ satisfy \eqref{1113},
$\omega\in A_{r,q}^+(\gamma)$,
and $(f,g)\in\mathcal{F}$ with $f\in L^q(\omega^q)$.
Without loss of generality, we
may assume that $\|f\|_{L^q(\omega^q)},
\|g\|_{L^r(\omega^r)}\in(0,\infty)$.
Let $s\in[1,\min\{q_0,\,q\}]$ satisfy
\begin{align}\label{1412}
\frac{1}{s'}=\frac{1}{r}-
\frac{1}{q}=\frac{1}{r_0}-\frac{1}{q_0}
\end{align}
and define
\begin{align*}
h_1:=\frac{|f|}{\|f\|_{L^q(\omega^q)}}+
\frac{|g|^{\frac{r}{q}}\omega^{\frac{r}{q}-1}}
{\|g\|_{L^r(\omega^r)}^{\frac{r}{q}}}.
\end{align*}
Then $\|h_1^s\|_{L^{1+\frac{q}{r'}}(\omega^q)}
^\frac{1}{s}=\|h_1\|_{L^q(\omega^q)}\leq 2$.
Applying Definition \ref{parabolic weight} and Remark
\ref{parabolic weight remark}, we obtain $\omega^q\in
A_{1+\frac{q}{r'}}^+(\gamma)$ with
$[\omega^q]_{A_{1+\frac{q}{r'}}^+(\gamma)}=
[\omega]_{A_{r,q}^+(\gamma)}^q$, which, together with Lemma
\ref{weighted inequality uncentered}(i), further implies that
$M^{\gamma+}$ is bounded on $L^{1+\frac{q}{r'}}
(\omega^q)$. Next, we introduce the following parabolic Rubio de
Francia iteration algorithm $H_1$ for $h_1$: let
\begin{align*}
H_1:=\left[\sum_{k\in\mathbb{Z}_+}\frac{(M^{\gamma+})^k(h_1^s)}
{2^k\|M^{\gamma+}\|^k_{\mathscr{L}(L^{1+\frac{q}{r'}}
(\omega^q))}}\right]^{\frac{1}{s}},
\end{align*}
where $(M^{\gamma+})^0$ denotes the identity operator and
$\|M^{\gamma+}\|_{\mathscr{L}(L^{1+\frac{q}{r'}}
(\omega^q))}$ denotes the operator norm of
$M^{\gamma+}$ on $L^{1+\frac{q}{r'}}(\omega^q)
$. From the monotone convergence theorem and \cite[Proposition
2.5]{km(am-2024)}, we deduce that the following assertions hold.
\begin{enumerate}
\item[\rm(i)] $h_1\leq H_1$.

\item[\rm(ii)] $\|H_1\|_{L^q(\omega^q)}\leq
2^{\frac{1}{s}}\|h_1^s\|_{L^{1+\frac{q}{r'}}(\omega^q)}
^{\frac{1}{s}}\leq 2^{\frac{1}{s}+1}$.

\item[\rm(iii)] $H_1^s\in A_1^-(\gamma)$ and
$[H_1^s]_{A_1^-(\gamma)}\leq2
\|M^{\gamma+}\|_{\mathscr{L}(L^{1+\frac{q}{r'}}(\omega^q))}$.
\end{enumerate}

On the other hand, by duality, there exists a nonnegative
function $h_2\in L^{(1+\frac{q}{r'})'}(\omega^q)$ such that
$\|h_2\|_{L^{(1+\frac{q}{r'})'}(\omega^q)}=1$ and
\begin{align}\label{20241205.1547}
\|f\|_{L^q(\omega^q)}=
\left\|\,|f|^s\right\|_{L^{1+\frac{q}{r'}}(\omega^q)}
^{\frac{1}{s}}=\left(\int_{\mathbb{R}^{n+1}}
|f|^sh_2\omega^q\right)^{\frac{1}{s}}.
\end{align}
Using Definition \ref{parabolic weight} and Remark
\ref{parabolic weight remark}, we find that $\omega^{-r'}\in
A_{1+\frac{r'}{q}}^-(\gamma)$ and $[\omega^{-r'}]
_{A_{1+\frac{r'}{q}}^-(\gamma)}=
[\omega]_{A_{r,q}^+(\gamma)}^{r'}$, which, together with Lemma
\ref{weighted inequality uncentered}(i), further implies that
$M^{\gamma-}$ is bounded on $L^{1+\frac{r'}{q}}
(\omega^{-r'})$. Now, we introduce the following parabolic
Rubio de Francia iteration algorithm $H_2$ for $h_2$: let
\begin{align*}
H_2:=\sum_{k\in\mathbb{Z}_+}\frac{(M^{\gamma-})^k(h_2\omega^q)}
{2^k\|M^{\gamma-}\|^k_{\mathscr{L}(L^{1+\frac{r'}{q}}(\omega^{-r'}))}}\omega^{-q},
\end{align*}
where $(M^{\gamma-})^0$ denotes the identity operator and
$\|M^{\gamma-}\|_{\mathscr{L}(L^{1+\frac{r'}{q}}
(\omega^{-r'}))}$ denotes the operator norm of
$M^{\gamma-}$ on $L^{1+\frac{r'}{q}}(\omega^{-r'})$. From
$\|h_2\omega^q\|_{L^{(1+\frac{q}{r'})'}
(\omega^{-r'})}=\|h_2\|_{L^{(1+\frac{q}{r'})'}
(\omega^q)}=1$, $(1+\frac{q}{r'})'=
1+\frac{r'}{q}$, the monotone convergence theorem, and
\cite[Proposition 2.5]{km(am-2024)}, it follows that the
following statements hold.
\begin{enumerate}
\item[\rm(iv)] $h_2\leq H_2$.

\item[\rm(v)] $\|H_2\|_{L^{(1+\frac{q}{r'})'}(\omega^q)}
\leq2\|h_2\omega^q\|_{L^{(1+\frac{q}{r'})'}(\omega^{-r'})}=2$.

\item[\rm(vi)] $H_2\omega^q\in A_1^+(\gamma)$ and
$[H_2\omega^q]_{A_1^+(\gamma)}\leq2
\|M^{\gamma-}\|_{\mathscr{L}(L^{1+\frac{r'}{q}}(\omega^{-r'}))}$.
\end{enumerate}

To show \eqref{20241204.1152}, we consider the
following two cases for $r_0$.

\emph{Case 1)} $r_0=1$. In this case, \eqref{1412} implies that
$s=q_0$. In addition, according to H\"older's inequality and
(v), we obtain
\begin{align*}
\int_{\mathbb{R}^{n+1}}|f|^{q_0}H_2\omega^q\leq
\|f\|_{L^q(\omega^q)}^{q_0}\|H_2\|
_{L^{(1+\frac{q}{r'})'}(\omega^q)}<\infty.
\end{align*}
Combining this, \eqref{20241205.1547}, $s=q_0$, (iv),
\eqref{20241204.1150}, H\"older's inequality,
$(H_2\omega^q)^{\frac{1}{q_0}}\in A_{1,q_0}^+(\gamma)$ [which
can be deduced from (vi) and Definition \ref{parabolic
weight}(ii)], and (v), we conclude that
\begin{align*}
\|f\|_{L^q(\omega^q)}&=\left(\int_{\mathbb{R}^{n+1}}
|f|^{q_0}h_2\omega^q\right)^{\frac{1}{q_0}}\leq
\left(\int_{\mathbb{R}^{n+1}}
|f|^{q_0}H_2\omega^q\right)^{\frac{1}{q_0}}\\
&\lesssim\int_{\mathbb{R}^{n+1}}|g|(H_2\omega^q)
^{\frac{1}{q_0}}\leq\|g\|_{L^r(\omega^r)}
\|H_2\|_{L^{(1+\frac{q}{r'})'}(\omega^q)}
^{\frac{1}{q_0}}\lesssim\|g\|_{L^r(\omega^r)}
\end{align*}
and hence \eqref{20241204.1152} holds in this case.

\emph{Case 2)} $r_0\in(1,q_0]$. In this case, \eqref{1412}
implies that $s<\min\{q_0,\,q\}$. From \eqref{20241205.1547},
(iv), H\"older's inequality, (ii), and (v), it follows that
\begin{align}\label{20241205.1658}
\|f\|_{L^q(\omega^q)}&=\left(\int_{\mathbb{R}^{n+1}}
|f|^sh_2\omega^q\right)^{\frac{1}{s}}\leq
\left(\int_{\mathbb{R}^{n+1}}
|f|^sH_2\omega^q\right)^{\frac{1}{s}}\\
&=\left[\int_{\mathbb{R}^{n+1}}|f|^s
H_1^{-\frac{s}{(1+\frac{q_0}{r_0'})'}}
H_1^{\frac{s}{(1+\frac{q_0}{r_0'})'}}H_2\omega^q
\right]^{\frac{1}{s}}\notag\\
&\leq\left(\int_{\mathbb{R}^{n+1}}|f|^{q_0}H_1^{s-q_0}
H_2\omega^q\right)^{\frac{1}{q_0}}
\left(\int_{\mathbb{R}^{n+1}}H_1^sH_2\omega^q
\right)^{\frac{1}{s(1+\frac{q_0}{r_0'})'}}\notag\\
&\leq\left(\int_{\mathbb{R}^{n+1}}|f|^{q_0}H_1^{s-q_0}
H_2\omega^q\right)^{\frac{1}{q_0}}
\|H_1\|_{L^q(\omega^q)}
^{\frac{1}{s(1+\frac{q_0}{r_0'})'}}
\|H_2\|_{L^{(1+\frac{q}{r'})'}(\omega^q)}
^{\frac{1}{s(1+\frac{q_0}{r_0'})'}}\notag\\
&\lesssim\left(\int_{\mathbb{R}^{n+1}}|f|^{q_0}H_1^{s-q_0}
H_2\omega^q\right)^{\frac{1}{q_0}}.\notag
\end{align}
We first prove that $W:=(H_1^{s-q_0}H_2\omega^q)^\frac{1}{q_0}
\in A_{r_0,q_0}^+(\gamma)$. Indeed, by (iii) and (vi), we find
that, for any given $R\in\mathcal{R}_p^{n+1}$.
\begin{align*}
&\left[\fint_{R^-(\gamma)}W^{q_0}\right]^{\frac{1}{q_0}}
\left[\fint_{R^+(\gamma)}W^{-r_0'}\right]^{\frac{1}{r_0'}}\\
&\quad=\left[\fint_{R^-(\gamma)}H_1^{s-q_0}H_2\omega^q\right]
^{\frac{1}{q_0}}\left[\fint_{R^+(\gamma)}H_1^{\frac{(q_0-s)r_0'}
{q_0}}(H_2\omega^q)^{-\frac{r_0'}{q_0}}\right]^{\frac{1}{r_0'}}\\
&\quad\lesssim\left[\fint_{R^+(\gamma)}H_1^s\right]
^{\frac{s-q_0}{sq_0}}\left[\fint_{R^-(\gamma)}H_2\omega^q
\right]^{\frac{1}{q_0}}\left[\fint_{R^+(\gamma)}
H_1^{\frac{(q_0-s)r_0'}{q_0}}\right]
^{\frac{1}{r_0'}}\left[\fint_{R^-(\gamma)}H_2\omega^q\right]
^{-\frac{1}{q_0}}=1.
\end{align*}
Taking the supremum over all $R\in\mathcal{R}_p^{n+1}$, we obtain
$W\in A_{r_0,q_0}^+(\gamma)$. Next, from the fact that $H_1\ge
h_1\ge\frac{|f|}{L^q(\omega^q)}$, H\"older's inequality, (ii),
and (v), we infer that
\begin{align*}
\int_{\mathbb{R}^{n+1}}|f|^{q_0}W^{q_0}&\leq
\|f\|_{L^q(\omega^q)}^{q_0}
\int_{\mathbb{R}^{n+1}}H_1^sH_2\omega^q\\
&\leq\|f\|_{L^q(\omega^q)}^{q_0}
\|H_1\|_{L^q(\omega^q)}
\|H_2\|_{L^{(1+\frac{q}{r'})'}(\omega^q)}<\infty,
\end{align*}
which, together with \eqref{20241205.1658},
$W=(H_1^{s-q_0}H_2\omega^q)^{\frac{1}{q_0}}\in
A_{r_0,q_0}^+(\gamma)$, \eqref{20241204.1150},
the fact that $H_1\ge
h_1\ge\frac{|g|^{\frac{r}{q}}\omega^{\frac{r}{q}-1}}
{\|g\|_{L^r(\omega^r)}^{\frac{r}{q}}}$,
H\"older's inequality with exponent
$\frac{q_0(1+\frac{q}{r'})'}{r_0}\in(1,\infty)$,
\begin{align*}
\left[\frac{r_0q}{r}-\frac{r_0(q_0-s)}{q_0}\right]
\left[\frac{q_0(1+\frac{q}{r'})'}{r_0}\right]'=q,
\end{align*}
(ii), and (v), further implies that
\begin{align*}
\|f\|_{L^q(\omega^q)}&\lesssim
\left(\int_{\mathbb{R}^{n+1}}|f|^{q_0}W^{q_0}\right)
^{\frac{1}{q_0}}\lesssim\left(\int_{\mathbb{R}^{n+1}}
|g|^{r_0}W^{r_0}\right)^{\frac{1}{r_0}}\\
&=\left[\int_{\mathbb{R}^{n+1}}|g|^{r_0}
H_1^{\frac{r_0(s-q_0)}{q_0}}H_2^{\frac{r_0}{q_0}}
\omega^{\frac{r_0q}{q_0}}\right]^{\frac{1}{r_0}}\\
&\leq\|g\|_{L^r(\omega^r)}
\left[\int_{\mathbb{R}^{n+1}}
H_1^{\frac{r_0q}{r}-\frac{r_0(q_0-s)}{q_0}}H_2^{\frac{r_0}{q_0}}
\omega^q\right]^{\frac{1}{r_0}}\\
&\leq\|g\|_{L^r(\omega^r)}
\|H_1\|_{L^q(\omega^q)}
^{\frac{q}{r_0[\frac{q_0(1+\frac{q}{r'})'}{r_0}]'}}
\|H_2\|_{L^{(1+\frac{q}{r'})'}(\omega^q)}
^{\frac{1}{q_0}}\lesssim\|g\|_{L^r(\omega^r)}
\end{align*}
and hence \eqref{20241204.1152} holds in this case.

Finally, combining the above two cases then
completes the proof of Theorem \ref{parabolic extrapolation}.
\end{proof}

Now, we give three extensions of Theorem
\ref{parabolic extrapolation}. The first one is the following
weak type parabolic extrapolation.

\begin{corollary}\label{parabolic extrapotion cor 1}
Let $\gamma\in(0,1)$ and $1\leq r_0\leq q_0<\infty$. Let
$\mathcal{F}$ be the set of all pairs $(f,g)$ of
measurable functions on $\mathbb{R}^{n+1}$ such that,
for any $\omega_0\in A_{r_0,q_0}^+(\gamma)$,
\begin{align}\label{20241206.2101}
\|f\|_{L^{q_0,\infty}(\omega_0^{q_0})}\leq C_0
\|g\|_{L^{r_0}(\omega_0^{r_0})}
\end{align}
holds provided that the left-hand side of
\eqref{20241206.2101} is finite, where the positive constant
$C_0$ in \eqref{20241206.2101} depends only on $\mathcal{F}$,
$n$, $p$, $\gamma$, $r_0$, $q_0$, and
$[\omega_0]_{A_{r_0,q_0}^+(\gamma)}$. Then, for any
$1<r\leq q<\infty$
satisfying $\frac{1}{r}-\frac{1}{q}=\frac{1}{r_0}-\frac{1}{q_0}$
and $\omega\in A_{r,q}^+(\gamma)$, there
exists a positive constant $C$, depending only on
$\mathcal{F}$, $n$, $p$, $\gamma$, $r_0$, $q_0$, $r$, $q$, and
$[\omega]_{A_{r,q}^+(\gamma)}$, such that, for any
$(f,g)\in\mathcal{F}$,
\begin{align}\label{20241206.2102}
\|f\|_{L^{q,\infty}(\omega^q)}\leq C
\|g\|_{L^r(\omega^r)}
\end{align}
holds provided that the left-hand side of
\eqref{20241206.2102} is finite.
\end{corollary}

\begin{proof}
Fix $1<r\leq q<\infty$ satisfying
$\frac{1}{r}-\frac{1}{q}=\frac{1}{r_0}-\frac{1}{q_0}$,
$\omega\in A_{r,q}^+(\gamma)$, and $(f,g)\in\mathcal{F}$ with
$f\in L^{q,\infty}(\omega^q)$. Let
\begin{align*}
\widetilde{\mathcal{F}}:=\left\{\left(\lambda\boldsymbol{1}
_{\{|u|>\lambda\}},v\right):\ (u,v)\in\mathcal{F},
\ \lambda\in(0,\infty)\right\}.
\end{align*}
From \eqref{20241206.2101} and the definition of
$\mathcal{F}$, we deduce that, for any $\lambda\in(0,\infty)$,
$\omega_0\in A_{r_0,q_0}^+(\gamma)$, and $(u,v)\in\mathcal{F}$
with $\|u\|_{L^{q_0,\infty}(\omega_0^{q_0})}<\infty$,
\begin{align*}
\left\|\lambda\boldsymbol{1}_{\{|u|>\lambda\}}\right\|
_{L^{q_0}(\omega_0^{q_0})}=
\lambda\left[\left(\omega_0^{q_0}\right)(\{|u|>\lambda\})\right]
^{\frac{1}{q_0}}\leq\|u\|_{L^{q_0,\infty}(\omega_0^{q_0})}\lesssim
\|v\|_{L^{r_0}(\omega_0^{r_0})}.
\end{align*}
Applying this and Theorem \ref{parabolic extrapolation},
we conclude that, for any $\lambda\in(0,\infty)$,
\begin{align*}
\lambda\left[\left(\omega^q\right)(\{|f|>\lambda\})\right]
^{\frac{1}{q}}=\left\|\lambda
\boldsymbol{1}_{\{|f|>\lambda\}}\right\|
_{L^q(\omega^q)}\lesssim
\|g\|_{L^r(\omega^r)}
\end{align*}
and hence \eqref{20241206.2102} holds by taking the supremum
over all $\lambda\in(0,\infty)$. This finishes the proof
of Corollary \ref{parabolic extrapotion cor 1}.
\end{proof}

Recall that, for any $\gamma\in[0,1)$, the
\emph{parabolic Muckenhoupt class $A_\infty^+(\gamma)$ with time
lag} is defined by setting
\begin{align*}
A_\infty^+(\gamma):=\bigcup_{q\in[1,\infty)}A_q^+(\gamma)
\end{align*}
and, for any $\omega\in A_\infty^+(\gamma)$, let
$[\omega]_{A_\infty^+(\gamma)}:=\lim_{q\to\infty}
[\omega]_{A_q^+(\gamma)}$. The second extension of Theorem
\ref{parabolic extrapolation} is the following parabolic
$A_\infty$ extrapolation.

\begin{corollary}\label{parabolic extrapotion cor 4}
Let $\gamma\in(0,1)$ and $q_0\in(0,\infty)$. Let $\mathcal{F}$
be the set of all pairs $(f,g)$ of measurable
functions on $\mathbb{R}^{n+1}$ such that, for any
$\omega_0\in A_\infty^+(\gamma)$,
\begin{align}\label{20240519.1601}
\|f\|_{L^{q_0}(\omega_0)}\leq C_0\|g\|_{L^{q_0}(\omega_0)}
\end{align}
holds provided that the left-hand of
\eqref{20240519.1601} is finite, where the positive constant
$C_0$ in \eqref{20240519.1601} depends only on $\mathcal{F}$,
$n$, $p$, $\gamma$, $q_0$, and
$[\omega_0]_{A_\infty^+(\gamma)}$. Then, for any
$q\in(0,\infty)$ and $\omega\in A_\infty^+(\gamma)$, there
exists a positive constant $C$, depending only on
$\mathcal{F}$, $n$, $p$, $\gamma$, $q_0$, $q$, and
$[\omega]_{A_\infty^+(\gamma)}$,
such that, for any $(f,g)\in\mathcal{F}$,
\begin{align}\label{20240519.1609}
\|f\|_{L^q(\omega)}\leq C\|g\|_{L^q(\omega)}
\end{align}
holds provided that the left-hand side of
\eqref{20240519.1609} is finite.
\end{corollary}

\begin{proof}
Fix $q\in(0,\infty)$, $\omega\in A_\infty^+(\gamma)$, and
$(f,g)\in\mathcal{F}$ with $f\in L^q(\omega)$.
Without loss of generality, we may assume that
$\|f\|_{L^q(\omega)},\|g\|_{L^q(\omega)}\in(0,\infty)$.
According to the definition of $A_\infty^+(\gamma)$ and
\cite[Proposition 3.4(i)]{ks(apde-2016)}, we find that there
exists $R\in(\max\{\frac{q}{q_0},\,1\},\infty)$ such that, for
any $r\in[R,\infty)$, $\omega\in A_r^+(\gamma)$. Fix such $r$
and let $\widetilde{\mathcal{F}}:=\{(|u|^{\frac{q_0}{r_0}},
|v|^{\frac{q_0}{r_0}}):\ (u,v)\in\mathcal{F}\}$, where
$r_0\in(1,\infty)$ is determined later. Using
\eqref{20240519.1601}, we conclude that, for any $\omega_0\in
A_{r_0}^+(\gamma)$ and $(u,v)\in\mathcal{F}$,
\begin{align*}
\left\|\left|u\right|^{\frac{q_0}{r_0}}\right\|^{r_0}
_{L^{r_0}(\omega_0)}=\int_{\mathbb{R}^{n+1}}\left|u\right|^{q_0}\omega_0
\lesssim\int_{\mathbb{R}^{n+1}}\left|v\right|^{q_0}\omega_0=
\left\|\left|v\right|^{\frac{q_0}{r_0}}\right\|_{L^{r_0}(\omega_0)}^{r_0}.
\end{align*}
From this and Theorem \ref{parabolic extrapolation}, we infer
that
\begin{align}\label{20240519.1744}
\left\|\left|f\right|^{\frac{q_0}{r_0}}\right\|_{L^r(\omega)}\lesssim
\left\|\left|g\right|^{\frac{q_0}{r_0}}\right\|_{L^r(\omega)}.
\end{align}
Finally, let $r_0:=\frac{q_0r}{q}$. By the arbitrariness of
$r\in[R,\infty)$, the definition of
$[\omega]_{A_\infty^+(\gamma)}$,
and \eqref{20240519.1744}, we obtain \eqref{20240519.1609},
which completes the proof of Corollary
\ref{parabolic extrapotion cor 4}.
\end{proof}

The third extension of Theorem \ref{parabolic extrapolation} is
the following vector-valued version of the
parabolic extrapolation.

\begin{corollary}\label{parabolic extrapotion cor 2}
Let $\gamma\in(0,1)$ and $1\leq r_0\leq q_0<\infty$. Let
$\mathcal{F}$ be the set of all pairs $(f,g)$ of
measurable functions on $\mathbb{R}^{n+1}$ such that,
for any $\omega_0\in A_{r_0,q_0}^+(\gamma)$,
\begin{align}\label{20241207.2044}
\|f\|_{L^{q_0}(\omega_0^{q_0})}\leq
C_0\|g\|_{L^{r_0}(\omega_0^{r_0})}
\end{align}
holds provided that the left-hand side of
\eqref{20241207.2044} is finite, where the positive constant
$C_0$ in \eqref{20241207.2044} depends only on $\mathcal{F}$,
$n$, $p$, $\gamma$, $r_0$, $q_0$, and
$[\omega_0]_{A_{r_0,q_0}^+(\gamma)}$. Then, for any
$1<r\leq q<\infty$ satisfying
$\frac{1}{r}-\frac{1}{q}=\frac{1}{r_0}-\frac{1}{q_0}$,
$s\in(1,\infty)$, and $\omega\in A_{r,q}^+(\gamma)$, there
exists a positive constant $C$, depending only on $\mathcal{F}$,
$n$, $p$, $\gamma$, $r_0$, $q_0$, $r$, $q$, $s$, and
$[\omega]_{A_{r,q}^+(\gamma)}$, such that, for any sequence
$\{(f_k,g_k)\}_{k\in\mathbb{N}}$ in $\mathcal{F}$,
\begin{align}\label{20240517.1757}
\left\|\left(\sum_{k\in\mathbb{N}}\left|f_k\right|^s\right)
^{\frac{1}{s}}\right\|_{L^q(\omega^q)}\leq
C\left\|\left(\sum_{k\in\mathbb{N}}\left|g_k\right|^s\right)
^{\frac{1}{s}}\right\|_{L^r(\omega^r)}
\end{align}
holds provided that the left-hand side of
\eqref{20240517.1757} is finite.
\end{corollary}

\begin{proof}
Fix $1<r\leq q<\infty$ satisfying $\frac{1}{r}-\frac{1}{q}=
\frac{1}{r_0}-\frac{1}{q_0}$, $s\in(1,\infty)$, $\omega\in
A_{r,q}^+(\gamma)$, and $\{(f_k,g_k)\}_{k\in\mathbb{N}}$
in $\mathcal{F}$ satisfying
$(\sum_{k\in\mathbb{N}}|f_k|^s)^{\frac{1}{s}}\in
L^q(\omega^q)$. Let
\begin{align*}
\widetilde{\mathcal{F}}:=\left\{\left(\left[\sum_{k=1}^N
\left|u_k\right|^s\right]^{\frac{1}{s}},\left[\sum_{k=1}^N
\left|v_k\right|^s\right]^{\frac{1}{s}}\right):\
N\in\mathbb{N},\ \{(u_k,v_k)\}_{k=1}^N\text{ in }\mathcal{F}\right\}.
\end{align*}
Since $s\in(1,\infty)$, we can choose $1<r_1\leq q_1<\infty$
such that $r_1\leq s\leq q_1$ and
$\frac{1}{r_1}-\frac{1}{q_1}=\frac{1}{r_0}-\frac{1}{q_0}$.
Applying \eqref{20241207.2044} and Theorem
\ref{parabolic extrapolation}, we conclude that, for
any $\omega_1\in A_{r_1,q_1}^+(\gamma)$ and
$(u,v)\in\mathcal{F}$ with $u\in L^{q_1}(\omega_1^{q_1})$,
\begin{align*}
\|u\|_{L^{q_1}(\omega_1^{q_1})}\lesssim
\|v\|_{L^{r_1}(\omega_1^{r_1})},
\end{align*}
which, together with $r_1\leq s\leq q_1$ and
Minkowski's integral inequality, further
implies that, for any $N\in\mathbb{N}$ and
$([\sum_{k=1}^N|u_k|^s]^{\frac{1}{s}},[\sum_{k=1}^N|v_k|^s]
^\frac{1}{s})\in\widetilde{\mathcal{F}}$ with $\{u_k\}_{k=1}
^N$ in $L^{q_1}(\omega_1^{q_1})$,
\begin{align*}
\left\|\left(\sum_{k=1}^N\left|u_k\right|^s\right)
^\frac{1}{s}\right\|_{L^{q_1}(\omega_1^{q_1})}^s&\leq
\sum_{k=1}^N\|u_k\|_{L^{q_1}(\omega_1^{q_1})}^s\\
&\lesssim\sum_{k=1}^N\|v_k\|_{L^{r_1}(\omega_1^{r_1})}^s
\leq\left\|\left(\sum_{k=1}^N\left|v_k\right|^s\right)
^\frac{1}{s}\right\|_{L^{r_1}(\omega_1^{r_1})}^s.
\end{align*}
From this and Theorem \ref{parabolic extrapolation},
it follows that, for any $N\in\mathbb{N}$,
\begin{align*}
\left\|\left(\sum_{k=1}^N\left|f_k\right|^s\right)
^\frac{1}{s}\right\|_{L^q(\omega^q)}\lesssim
\left\|\left(\sum_{k=1}^N\left|g_k\right|^s\right)
^\frac{1}{s}\right\|_{L^r(\omega^r)}.
\end{align*}
This, combined with the monotone convergence theorem, further
implies that \eqref{20240517.1757} holds, which completes the
proof of Corollary \ref{parabolic extrapotion cor 2}.
\end{proof}

As an application of both Lemma
\ref{weighted inequality uncentered}(i) and Corollary
\ref{parabolic extrapotion cor 2}, we obtain the following
parabolic weighted Fefferman--Stein vector-valued inequality;
we omit the details.

\begin{corollary}\label{parabolic extrapotion cor 3}
Let $\gamma\in(0,1)$, $\alpha\in[0,1)$, $1<r\leq q<\infty$
with $\frac{1}{r}-\frac{1}{q}=\alpha$, $s\in(1,\infty)$, and
$\omega\in A_{r,q}^+(\gamma)$. Then there exists a positive
constant $C$, depending only on $n$, $p$, $\gamma$, $\alpha$,
$q$, $r$, $s$, and $[\omega]_{A_{r,q}^+(\gamma)}$, such that,
for any sequence $\{f_k\}_{k\in\mathbb{N}}$ in
$L_{\mathrm{loc}}^1$,
\begin{align*}
\left\|\left\{\sum_{k\in\mathbb{N}}\left[M^{\gamma+}_\alpha
(f_k)\right]^s\right\}^\frac{1}{s}\right\|_{L^q(\omega^q)}\leq C
\left\|\left(\sum_{k\in\mathbb{N}}\left|f_k\right|^s\right)
^\frac{1}{s}\right\|_{L^r(\omega^r)}.
\end{align*}
\end{corollary}

\subsection{Parabolic Extrapolation at Infinity}
\label{subsection2.2}

Denote by $\mathscr{M}$ the set of all measurable functions on
$\mathbb{R}^{n+1}$. Let $T$ be an operator defined on a linear
subspace $D(T)$ of $L_{\mathrm{loc}}^1$ and taking values in
$\mathscr{M}$. Then $T$ is said to be \emph{sublinear} if, for
any $f,g\in D(T)$ and $\lambda\in\mathbb{C}$,
$|T(f+g)|\leq|T(f)|+|T(g)|$ and $|T(\lambda
f)|=|\lambda|\,|T(f)|$. Let $C_\mathrm{c}$ be the set of all
continuous functions on $\mathbb{R}^{n+1}$ with compact support.
Then the parabolic Rubio de Francia extrapolation at infinity
can be stated as follows.

\begin{theorem}\label{parabolic extrapolation infinity}
Let $\gamma\in[0,1)$, $1\leq \tau<r_0\leq\infty$, and
$T:\ C_\mathrm{c}\to\mathscr{M}$ be a sublinear operator. If,
for any pairs $(u_0,v_0)$ of nonnegative functions
on $\mathbb{R}^{n+1}$ with $(u_0^\tau,v_0^\tau)\in
TA_{\frac{r_0}{\tau},\infty}^+(\gamma)$, there exists a
positive constant $C_0$, depending only on $n$, $p$, $\gamma$,
$\tau$, $r_0$, and $[u_0^\tau,v_0^\tau]
_{TA_{\frac{r_0}{\tau},\infty}^+(\gamma)}$, such that, for any
$f\in C_\mathrm{c}$,
\begin{align}\label{20241206.1932}
\|T(f)u_0\|_{L^\infty}\leq C_0
\|fv_0\|_{L^{r_0}},
\end{align}
then, for any $r\in(\tau,r_0)$, $q\in(r,\infty)$ with
$\frac{1}{q}=\frac{1}{r}-\frac{1}{r_0}$, and any
nonnegative function $\omega$ on $\mathbb{R}^{n+1}$ with
$\omega^\tau\in A_{\frac{r}{\tau},\frac{q}{\tau}}^+(\gamma)$,
there exists a positive constant $C$, depending only on $n$,
$p$, $\gamma$, $\tau$, $r_0$, $r$, $q$, and $[\omega^\tau]
_{A_{\frac{r}{\tau},\frac{q}{\tau}}^+(\gamma)}$, such that,
for any $f\in C_\mathrm{c}$,
\begin{align}\label{20241206.2039}
\|T(f)\|_{L^{q,\infty}(\omega^q)}\leq C\|f\|_{L^r(\omega^r)}.
\end{align}
\end{theorem}

To show Theorem \ref{parabolic extrapolation infinity}, we
need to recall weighted Lorentz spaces. Let $0<r,q\leq\infty$
and $\omega$ be a weight. Recall that the \emph{weighted Lorentz
space $L^{r,q}(\omega)$} is defined to be the set of
all $f\in\mathscr{M}$ such that
\begin{align*}
\|f\|_{L^{r,q}(\omega)}:=
\begin{cases}\displaystyle
\left[\int_0^\infty\left\{t^{\frac{1}{r}}f_\omega^*(t)\right\}^q
\frac{dt}{t}\right]^\frac{1}{q}&\mbox{if\ }q\in(0,\infty),\\
\sup\limits_{t\in(0,\infty)}t^{\frac{1}{r}}f_\omega^*(t)
&\mbox{if\ }q=\infty
\end{cases}
\end{align*}
is finite, where $f_\omega^*$ is the \emph{weighted decreasing
rearrangement} of $f$ defined by setting, for any
$t\in(0,\infty)$,
\begin{align*}
f_\omega^*(t):=\inf\left\{s\in(0,\infty):\ \omega(\{|f|>s\})
\leq t\right\}
\end{align*}
with the convention that $\inf\emptyset=\infty$. By borrowing
some ideas from the proof of \cite[Theorem 1.4]{mr(prse-2000)},
we are ready to prove Theorem
\ref{parabolic extrapolation infinity}.

\begin{proof}[Proof of Theorem \ref{parabolic extrapolation infinity}]
Let $r\in(\tau,r_0)$, $q\in(r,\infty)$ satisfy
$\frac{1}{q}=\frac{1}{r}-\frac{1}{r_0}$, $\omega$ be a
nonnegative function on $\mathbb{R}^{n+1}$ with
$\omega^\tau\in A_{\frac{r}{\tau},\frac{q}{\tau}}^+(\gamma)$,
and $f\in C_\mathrm{c}$. We first construct a pair  $(u,v)$ of
nonnegative functions on $\mathbb{R}^{n+1}$ such that
$(u^\tau,v^\tau)\in TA_{\frac{r_0}{\tau},\infty}^+(\gamma)$.
To do this, for any $(x,t)\in\mathbb{R}^{n+1}$, let
\begin{align*}
v(x,t):=
\begin{cases}
|f(x,t)|^{\frac{r}{r_0}-1}\left[\omega(x,t)\right]^\frac{r}{r_0}
\|f\|_{L^r(\omega^r)}^{\frac{r}{q}}&\mbox{if\ }|f(x,t)|\in
(0,\infty),\\
e^{\frac{\pi(|x|^2+t^2)}{q}}\omega(x,t)&\mbox{if\ }|f(x,t)|=0.
\end{cases}
\end{align*}
Using the definition of $v$, we find that
$\|f\|_{L^r(\omega^r)}=\|fv\|_{L^{r_0}}$ and
\begin{align}\label{20241205.1743}
\int_{\mathbb{R}^{n+1}}v^{-q}\omega^{q}\leq2.
\end{align}
Let
$u:=[M^{\gamma+}(v^{-\tau(\frac{r_0}{\tau})'})
]^{-\frac{1}{\tau(\frac{r_0}{\tau})'}}$.
We can easily verify that $(u^\tau,v^\tau)\in
TA_{\frac{r_0}{\tau},\infty}^+(\gamma)$ by Definition
\ref{parabolic maximal operator}. From H\"older's inequality
for weighted Lorentz spaces (see, for example,
\cite[(6.2.61)]{Mitrea-bookI}), we deduce that, for any
given $\lambda\in(0,\infty)$ and any compact set
$\mathcal{K}\subset\mathbb{R}^{n+1}$,
\begin{align}\label{20241206.1949}
(\omega^q)(\mathcal{K}\cap\{|T(f)|>\lambda\})
&=\int_{\mathbb{R}^{n+1}}
\boldsymbol{1}_{\mathcal{K}\cap\{|T(f)|>\lambda\}}
u^{-1}u\omega^q\\
&\leq\left\|\boldsymbol{1}_{\mathcal{K}\cap\{|T(f)|>\lambda\}}
\right\|_{L^{\frac{1}{q}+1,1}(u\omega^q)}
\left\|u^{-1}\right\|_{L^{q+1,\infty}
(u\omega^q)}\notag\\
&=:\mathrm{I}\times\mathrm{II}.\notag
\end{align}
We first estimate $\mathrm{I}$. By a simple computation, we obtain
\begin{align*}
\left(\boldsymbol{1}_{\{|T(f)|>\lambda\}}\right)_{u\omega^q}^*=
\boldsymbol{1}_{[0,\Lambda)},
\end{align*}
where $\Lambda:=(u\omega^q)(\mathcal{K}\cap\{|T(f)|>\lambda\})$.
This, together with Chebyshev's inequality,
\eqref{20241206.1932}, and $(u^\tau,v^\tau)\in
TA_{\frac{r_0}{\tau},\infty}^+(\gamma)$, further implies that
\begin{align}\label{20241206.1950}
\mathrm{I}&=\int_0^\infty t^{\frac{q}{q+1}}
\boldsymbol{1}_{[0,\Lambda)}(t)\frac{dt}{t}\sim
\left[\int_{\mathcal{K}\cap\{|T(f)|>\lambda\}}u\omega^q\right]
^\frac{q}{q+1}\\
&\leq\left[\frac{1}{\lambda}
\int_{\mathcal{K}\cap\{|T(f)|>\lambda\}}
|T(f)|u\omega^q\right]^\frac{q}{q+1}\notag\\
&\leq\left[\frac{1}{\lambda}\|T(f)u\|
_{L^\infty}(\omega^q)(\mathcal{K}\cap\{|T(f)|>\lambda\})
\right]^\frac{q}{q+1}\notag\\
&\lesssim\left[\frac{1}{\lambda}\|fv\|
_{L^{r_0}}(\omega^q)(\mathcal{K}\cap\{|T(f)|>\lambda\})
\right]^\frac{q}{q+1}.\notag
\end{align}

Next, we estimate $\mathrm{II}$. From $\omega^\tau\in
A_{\frac{r}{\tau},\frac{q}{\tau}}^+(\gamma)$, Definition
\ref{parabolic weight}, and Remark
\ref{parabolic weight remark}, it follows that $\omega^q\in
A_{1+\frac{q/\tau}{(r/\tau)'}}^+(\gamma)$ and $[\omega^q]
_{A_{1+\frac{q/\tau}{(r/\tau)'}}^+(\gamma)}=[\omega^\tau]
_{A_{\frac{r}{\tau},\frac{q}{\tau}}^+(\gamma)}^\frac{q}{\tau}$,
which, together with Lemma
\ref{weighted inequality uncentered}(ii), further implies that
$M^{\gamma+}$ is bounded from $L^{1+\frac{q/\tau}{(r/\tau)'}}
(\omega^q)$ to $L^{1+\frac{q/\tau}{(r/\tau)'},\infty}(\omega^q)$.
Combining this, the definition of $u$, and
\eqref{20241205.1743}, we find that, for any $s\in(0,\infty)$,
\begin{align*}
s^{q+1}(u\omega^q)\left(\left\{u^{-1}>s\right\}\right)&=s^{q+1}
\int_{\{u^{-1}>s\}}u\omega^q\leq s^q\int_{\{u^{-1}>s\}}\omega^q\\
&=s^q(\omega^q)\left(\left\{M^{\gamma+}
\left(v^{-\tau(\frac{r_0}{\tau})'}\right)>
s^{\tau(\frac{r_0}{\tau})'}\right\}\right)\\
&\lesssim\frac{s^q}{s^{\tau(\frac{r_0}{\tau})'
[1+\frac{q/\tau}{(r/\tau)'}]}}\int_{\mathbb{R}^{n+1}}
v^{-\tau(\frac{r_0}{\tau})'[1+\frac{q/\tau}{(r/\tau)'}]}\omega^q\\
&=\int_{\mathbb{R}^{n+1}}v^{-q}\omega^q\leq2.
\end{align*}
Therefore, for any $t\in(0,\infty)$,
\begin{align*}
\left(u^{-1}\right)_{u\omega^q}^*(t)&=\inf\left\{s\in(0,\infty):\
(u\omega^q)\left(\left\{u^{-1}>s\right\}\right)\leq t\right\}\\
&\lesssim\inf\left\{s\in(0,\infty):\
\frac{1}{s^{q+1}}\leq t\right\}=\frac{1}{t^{\frac{1}{q+1}}}
\end{align*}
and hence
\begin{align*}
\mathrm{II}=\sup\limits_{t\in(0,\infty)}t^\frac{1}{q+1}
\left(u^{-1}\right)_{u\omega^q}^*(t)\lesssim1.
\end{align*}
This, together with \eqref{20241206.1949},
\eqref{20241206.1950}, and
$\|f\|_{L^r(\omega^r)}=
\|fv\|_{L^{r_0}(\mathbb{R}^{n+1})}$, further implies that
\begin{align*}
(\omega^q)(\mathcal{K}\cap\{|T(f)|>\lambda\})\lesssim
\frac{1}{\lambda^{\frac{q}{q+1}}}\|f\|
_{L^r(\omega^r)}^{\frac{q}{q+1}}
\left[(\omega^q)(\mathcal{K}\cap\{|T(f)|>\lambda\})\right]
^\frac{q}{q+1}.
\end{align*}
Observe that $(\omega^q)(\mathcal{K}\cap\{|T(f)|>\lambda\})
<\infty$. Consequently,
\begin{align*}
(\omega^q)(\mathcal{K}\cap\{|T(f)|>\lambda\})\lesssim
\frac{1}{\lambda^q}\|f\|_{L^r(\omega^r)}^q.
\end{align*}
Taking the suprema over all $\mathcal{K}\subset\mathbb{R}^{n+1}$
and then all $\lambda\in(0,\infty)$, we obtain
\eqref{20241206.2039}, which completes the proof of Theorem
\ref{parabolic extrapolation infinity}.
\end{proof}

\subsection{Parabolic Extrapolation for Commutators of Linear Operators}
\label{subsection2.3}

Let $T$ be a linear operator defined on a linear subspace $D(T)$
of $L_{\mathrm{loc}}^1$ and taking values in
$\mathscr{M}$ and let $b\in L_{\mathrm{loc}}^1$. Recall that the
\emph{(first order) commutator $[b,T]$ of $T$ with $b$} is
defined by setting, for any suitable function $f$ on
$\mathbb{R}^{n+1}$ and for any $(x,t)\in\mathbb{R}^{n+1}$,
\begin{align}\label{20250304.2137}
[b,T](f)(x,t):=[b,T]_1(f)(x,t):=b(x,t)T(f)(x,t)-T(bf)(x,t).
\end{align}
For any $k\in\mathbb{N}\cap[2,\infty)$,
the \emph{$k$th-order commutator $[b,T]_k$ of $T$ with $b$} is
defined by the recursive formula $[b,T]_k:=[b,[b,T]_{k-1}]$.
We can easily verify that, for any given $k\in\mathbb{N}$ and
for any suitable $f$ on
$\mathbb{R}^{n+1}$ and any $(x,t)\in\mathbb{R}^{n+1}$,
\begin{align}\label{20250304.2138}
[b,T]_k(f)(x,t)=T\left([b(x,t)-b(\cdot)]^kf(\cdot)\right)(x,t).
\end{align}
We then have the following
parabolic extrapolation for commutators of linear operators.

\begin{theorem}\label{parabolic extrapolation for commutator}
Let $\gamma\in(0,1)$, $1<r_0\leq q_0<\infty$, and $T:\
L_{\mathrm{loc}}^1\to\mathscr{M}$ be a linear operator. If, for
any $\omega_0\in A_{r_0,q_0}^+(\gamma)$, $T$ is bounded from
$L^{r_0}(\omega_0^{r_0})$ to $L^{q_0}(\omega_0^{q_0})$ with its
operator norm depending only on $n$, $p$, $\gamma$, $r_0$,
$q_0$, and $[\omega_0]_{A_{r_0,q_0}^+(\gamma)}$, then, for
any $k\in\mathbb{N}$, $1<r\leq q<\infty$ with $\frac{1}{r}-
\frac{1}{q}=\frac{1}{r_0}-\frac{1}{q_0}$, $\omega\in
A_{r,q}^+(\gamma)$, and $b\in\mathrm{PBMO}$, $[b,T]_k$ is
bounded from $L^r(\omega^r)$ to $L^q(\omega^q)$. Moreover, for
any $f\in L^r(\omega^r)$,
\begin{align*}
\|[b,T]_k(f)\|_{L^q(\omega^q)}\lesssim\|b\|
_{\mathrm{PBMO}}^k\|f\|_{L^r(\omega^r)},
\end{align*}
where the implicit positive constant is independent of $f$ and $b$.
\end{theorem}

In order to show Theorem
\ref{parabolic extrapolation for commutator}, we need
the following several lemmas.

\begin{lemma}\label{improving weight exponent}
Let $\gamma\in(0,1)$, $1<r\leq q<\infty$, and $\omega\in
A_{r,q}^+(\gamma)$. Then there exists $\epsilon_0\in(1,\infty)$,
depending only on $n$, $p$, $\gamma$, $r$, $q$, and
$[\omega]_{A_{r,q}^+(\gamma)}$, such that, for any
$\epsilon\in[1,\epsilon_0]$, $\omega^\epsilon\in
A_{r,q}^+(\gamma)$.
\end{lemma}

\begin{proof}
Using Definition \ref{parabolic weight} and Remark
\ref{parabolic weight remark}, we find that $\omega^q\in
A_{1+\frac{q}{r'}}^+(\gamma)$ and $\omega^{-r'}\in
A_{1+\frac{r'}{q}}^-(\gamma)$, which, together with
\cite[Theorem 3.1]{km(am-2024)}, further implies that
$\omega^q\in A_{1+\frac{q}{r'}}^+(\alpha)$ and $\omega^{-r'}\in
A_{1+\frac{r'}{q}}^-(\alpha)$ for any $\alpha\in(0,1)$ with
\begin{align*}
\left[\omega^q\right]_{A_{1+\frac{q}{r'}}^+(\gamma)}
\sim\left[\omega^q\right]_{A_{1+\frac{q}{r'}}^+(\alpha)}
\mbox{\ \ and\ \ }\left[\omega^{-r'}\right]
_{A_{1+\frac{r'}{q}}^-(\gamma)}\sim\left[\omega^{-r'}\right]
_{A_{1+\frac{r'}{q}}^-(\alpha)}.
\end{align*}
Thus, we may assume that $\gamma\in(\frac{1}{2},1)$ without loss
of generality. From \cite[Corollary 5.3]{km(am-2024)},
we infer that there exists $\epsilon_0\in(1,\infty)$ such that,
for any $\epsilon\in(1,\epsilon_0]$ and
$R:=R(x,t,L)\in\mathcal{R}_p^{n+1}$ with
$(x,t)\in\mathbb{R}^{n+1}$ and $L\in(0,\infty)$,
\begin{align*}
\left[\fint_{R^-(\gamma)}\omega^{\epsilon q}\right]
^\frac{1}{\epsilon}\lesssim\fint_{R^-(\gamma)
+(\mathbf{0},(1-\gamma)L^p)}\omega^q
\end{align*}
and
\begin{align*}
\left[\fint_{R^+(\gamma)}\omega^{-\epsilon r'}\right]
^\frac{1}{\epsilon}\lesssim\fint_{R^+(\gamma)
-(\mathbf{0},(1-\gamma)L^p)}\omega^{-r'},
\end{align*}
where the implicit positive constants are independent of
$R$. Applying this, both
\begin{align*}
R^-(\gamma)+(\mathbf{0},(1-\gamma)L^p)
\subset R^-(2\gamma-1)\ \mathrm{with\ equivalent\ measures}
\end{align*}
and $R^+(\gamma)-(\mathbf{0},
(1-\gamma)L^p)\subset R^+(2\gamma-1)$ with equivalent measures,
and \cite[Theorem 3.1]{km(am-2024)}, we conclude that
\begin{align*}
&\fint_{R^-(\gamma)}\omega^{\epsilon q}\left[\fint_{R^+(\gamma)}
\omega^{-\epsilon r'}\right]^{\frac{1}{r'}}\\
&\quad\lesssim\left\{\fint_{R^-(\gamma)+(\mathbf{0},(1-\gamma)L^p)}
\omega^q\left[\fint_{R^+(\gamma)-(\mathbf{0},(1-\gamma)L^p)}
\omega^{-r'}\right]^\frac{1}{r'}\right\}^\epsilon\\
&\quad\lesssim\left\{\fint_{R^-(2\gamma-1)}\omega^q\left[
\fint_{R^+(2\gamma-1)}\omega^{-r'}\right]^\frac{1}{r'}
\right\}^\epsilon
\leq[\omega]_{A_{r,q}^+(2\gamma-1)}^\epsilon<\infty.
\end{align*}
Taking the supremum over all $R\in\mathcal{R}_p^{n+1}$, we
obtain $\omega^\epsilon\in A_{r,q}^+(\gamma)$, which
completes the proof of Lemma \ref{improving weight exponent}.
\end{proof}

To prove the following lemma, we borrow
some ideas from \cite{abkp(sm-1993), bmmst(ma-2020)}. Recall
that, for any $z\in\mathbb{C}$, let $\Re(z)$ denote its real part.

\begin{lemma}\label{analytic}
Let $\gamma\in(0,1)$, $1<r\leq q<\infty$, and $T:\
L_{\mathrm{loc}}^1\to\mathscr{M}$ be a linear operator. If, for
any $\omega\in A_{r,q}^+(\gamma)$, $T$ is bounded from
$L^r(\omega^r)$ to $L^q(\omega^q)$ with its operator norm
depending only on $n$, $p$, $\gamma$, $r$, $q$, and
$[\omega]_{A_{r,q}^+(\gamma)}$, then, for any $\omega\in
A_{r,q}^+(\gamma)$ and $b\in\mathrm{PBMO}$, there exists
$\eta\in(0,\infty)$ such that, for any $f\in C_\mathrm{c}$, the map
\begin{align*}
\Phi_z(f):\ \begin{cases}
\{z\in\mathbb{C}:\ |z|\in[0,\eta)\}\to L^q(\omega^q),\\
z\mapsto e^{zb}T(e^{-zb}f)
\end{cases}
\end{align*}
is continuous.
\end{lemma}

\begin{proof}
Let $\omega\in A_{r,q}^+(\gamma)$, $b\in\mathrm{PBMO}$, and
$f\in C_\mathrm{c}$. We claim that there exists a positive
constant $C_1$ such that
\begin{align}\label{20250106.1117}
\sup_{\{z\in\mathbb{C}:\ |z|\leq C_1\}}
\left[\omega e^{\Re(z)b}\right]_{A_{r,q}^+(\gamma)}<\infty.
\end{align}
Indeed, if $b$ is a constant function, then
\eqref{20250106.1117} holds automatically for any
$C_1\in(0,\infty)$. Hence we only need to
assume that $b$ is not a constant function or, equivalently,
$\|b\|_{\mathrm{PBMO}}\in(0,\infty)$. In this case, by an
argument similar to that used in the proof of \cite[Corollary
6.13]{duo-book}, we find that there exists $B\in(0,\infty)$,
depending only on $n$ and $p$, such that, for any
$\delta\in(0,\frac{B}{\|b\|_{\mathrm{PBMO}}}]$,
\begin{align}\label{20250105.1711}
\sup_{R\in\mathcal{R}_p^{n+1}}\fint_Re^{\delta|b-b_R|}<\infty.
\end{align}
In addition, applying Lemma \ref{improving weight exponent},
we conclude that there exists $\epsilon\in(1,\infty)$, depending
only on $n$, $p$, $\gamma$, $r$, $q$, and
$[\omega]_{A_{r,q}^+(\gamma)}$, such that $\omega^\epsilon\in
A_{r,q}^+(\gamma)$. Let $C_1:=\frac{B}{\epsilon'(q+r')
\|b\|_{\mathrm{PBMO}}}$ and $z\in\mathbb{C}$ satisfy
$|z|\leq C_1$. From  H\"older's inequality, $\omega^\epsilon\in
A_{r,q}^+(\gamma)$, and \eqref{20250105.1711}, it follows that,
for any $R\in\mathcal{R}_p^{n+1}$,
\begin{align*}
&\left[\fint_{R^-(\gamma)}\omega^qe^{q\Re(z)b}\right]^\frac{1}{q}
\left[\fint_{R^+(\gamma)}\omega^{-r'}e^{-r'\Re(z)b}\right]
^\frac{1}{r'}\\
&\quad\leq\left[\fint_{R^-(\gamma)}\omega^{\epsilon q}\right]
^{\frac{1}{\epsilon q}}\left[\fint_{R^-(\gamma)}
e^{\epsilon'q\Re(z)(b-b_R)}\right]^\frac{1}{\epsilon'q}\\
&\quad\quad\times
\left[\fint_{R^+(\gamma)}\omega^{-\epsilon r'}\right]
^{\frac{1}{\epsilon r'}}\left[\fint_{R^+(\gamma)}
e^{-\epsilon'r'\Re(z)(b-b_R)}\right]^\frac{1}{\epsilon'r'}\\
&\quad\lesssim\left[\fint_{R^-(\gamma)}
e^{\epsilon'q|\Re(z)|\,|b-b_R|}\right]^\frac{1}{\epsilon'q}
\left[\fint_{R^+(\gamma)}e^{\epsilon'r'|\Re(z)|\,|b-b_R|}\right]
^\frac{1}{\epsilon'r'}\lesssim1,
\end{align*}
where the implicit positive constants are independent of $z$.
Taking the suprema over all $R\in\mathcal{R}_p^{n+1}$ and
$z\in\mathbb{C}$ with $|z|\leq C_1$, we obtain
\eqref{20250106.1117}, which completes the proof of the
aforementioned claim.

Similarly, we find that there exists a positive constant $C_2$
such that
\begin{align}\label{20250106.1610}
\sup_{\{z\in\mathbb{C}:\ |z|\leq C_2\}}
\left[\omega e^{\Re(z)|b|}\right]_{A_{r,q}^+(\gamma)}<\infty.
\end{align}

Now, let $C_3:=\min\{C_1,\,C_2\}$. We show that, for any
$z\in\mathbb{C}$ satisfying $|z|\leq C_3$,
$\Phi_z(f)\in L^q(\omega^q)$.
Indeed, fix such $z$. Using \eqref{20250106.1117} and the
weighted boundedness of $T$, we conclude that
\begin{align}\label{20250106.2034}
\sup_{\{z\in\mathbb{C}:\ |z|\leq C_3\}}
\|\Phi_z(f)\|_{L^q(\omega^q)}&=
\sup_{\{z\in\mathbb{C}:\ |z|\leq C_3\}}
\left\|T(e^{-zb}f)\right\|
_{L^q(\omega^qe^{q\Re(z)b})}\\
&\lesssim\sup_{\{z\in\mathbb{C}:\ |z|\leq C_3\}}
\left\|e^{-zb}f\right\|
_{L^r(\omega^re^{r\Re(z)b})}=\|f\|_{L^r(\omega^r)}<\infty.\notag
\end{align}
Therefore, $\Phi_z(f)\in L^q(\omega^q)$ and hence $\Phi_z(f)$ is
well-defined.

Finally, let $\eta:=\frac{C_3}{2}$. We prove that $\Phi_z(f)$ is
continuous in $\{z\in\mathbb{C}:\ |z|<\eta\}$. Fix such $z$
and let $\{z_k\}_{k\in\mathbb{N}}$ in $\{z\in\mathbb{C}:\
|z|<\eta\}$ be such that $z_k\to z$ as $k\to\infty$. It remains
to show that $\|\Phi_{z_k}(f)-\Phi_z(f)\|_{L^q(\omega^q)}\to0$ as
$k\to\infty$. Indeed, for any $k\in\mathbb{N}$,
\begin{align}\label{20250106.1631}
\left\|\Phi_{z_k}(f)-\Phi_z(f)\right\|_{L^q(\omega^q)}^q
&\lesssim\int_{\mathbb{R}^{n+1}}
\left|T\left(\left(e^{-z_kb}-e^{-zb}\right)f\right)\right|^q
\omega^qe^{q\Re(z_k)b}\\
&\quad+\int_{\mathbb{R}^{n+1}}\left|T\left(e^{-zb}f\right)
\right|^q\omega^q\left|e^{z_kb}-e^{zb}\right|^q\notag\\
&=:\mathrm{I}+\mathrm{II}.\notag
\end{align}
We estimate $\mathrm{I}$ and $\mathrm{II}$ separately. We first
deal with $\mathrm{I}$. Notice that
\begin{align*}
\left|(e^{-z_kb}-e^{-zb})f\right|\omega
e^{\Re(z_k)b}\leq2\|f\|_{L^\infty}\omega
e^{2\eta|b|}\boldsymbol{1}_{\mathrm{supp}(f)}\in L^r
\end{align*}
[which can be deduced from \eqref{20250106.1610} and $r<q$].
Combining this, \eqref{20250106.1117}, the weighted
boundedness of $T$, and Lebesgue's dominated convergence
theorem, we find that
\begin{align}\label{20250106.1553}
\mathrm{I}\lesssim\int_{\mathbb{R}^{n+1}}
\left|\left(e^{-z_kb}-e^{-zb}\right)f\right|^r
\omega^re^{r\Re(z_k)b}\to0
\end{align}
as $k\to\infty$.

Next, we handle $\mathrm{II}$. Applying
\eqref{20250106.1610}, the weighted boundedness of $T$, and
$\omega e^{2\eta|b|}\boldsymbol{1}_{\mathrm{supp}(f)}\in L^r$,
we conclude that
\begin{align*}
\int_{\mathbb{R}^{n+1}}\left|T(e^{-zb}f)\right|^q
\omega^qe^{q\eta|b|}
&\lesssim\int_{\mathbb{R}^{n+1}}\left|e^{-zb}f\right|^r
\omega^re^{r\eta|b|}\leq\int_{\mathbb{R}^{n+1}}|f|^r\omega^r
e^{2r\eta|b|}\\
&\leq\|f\|_{L^\infty}^r\int_{\mathrm{supp}(f)}
\omega^re^{2r\eta|b|}<\infty.
\end{align*}
From this, $|T(e^{-zb}f)|\omega|e^{z_kb}-e^{zb}|
\leq2|T(e^{-zb}f)|\omega e^{\eta|b|}$, and Lebesgue's dominated
convergence theorem again, it follows that $\mathrm{II}\to0$ as
$k\to\infty$. Combining this, \eqref{20250106.1631}, and
\eqref{20250106.1553}, we obtain
$\|\Phi_{z_k}(f)-\Phi_z(f)\|_{L^q(\omega^q)}\to0$
as $k\to\infty$. This finishes the proof of Lemma \ref{analytic}.
\end{proof}

Theorem \ref{parabolic extrapolation for commutator} is a direct
consequence of Theorem \ref{parabolic extrapolation}, the
following result, and a standard density argument.

\begin{theorem}\label{commutator boundedness}
Let $\gamma\in(0,1)$, $1<r\leq q<\infty$, and $T:\
L_{\mathrm{loc}}^1\to\mathscr{M}$
be a linear operator. If, for any $\omega\in
A_{r,q}^+(\gamma)$, $T$ is bounded from
$L^r(\omega^r)$ to $L^q(\omega^q)$ with its operator norm
depending only on $n$, $p$, $\gamma$, $r$, $q$, and
$[\omega]_{A_{r,q}^+(\gamma)}$, then, for any  $k\in\mathbb{N}$,
$\omega\in A_{r,q}^+(\gamma)$, $b\in\mathrm{PBMO}$,
and $f\in C_{\mathrm{c}}$,
\begin{align}\label{20250106.2136}
\left\|[b,T]_k(f)\right\|_{L^q(\omega^q)}
\lesssim\|b\|_{\mathrm{PBMO}}^k
\|f\|_{L^r(\omega^r)},
\end{align}
where the implicit positive constant is independent of $f$ and $b$.
\end{theorem}

\begin{proof}
Let $k\in\mathbb{N}$, $\omega\in A_{r,q}^+(\gamma)$,
$b\in\mathrm{PBMO}$, and $f\in C_\mathrm{c}$.
For any $N\in\mathbb{N}$ and $z\in\mathbb{C}$, let
\begin{align*}
S_z^{(N)}:=\sum_{j=0}^N\frac{(zb)^j}{j!}\mbox{\ \ and\ \ }
\Phi_z^{(N)}(f):=S_z^{(N)}T\left(S_{-z}^{(N)}f\right).
\end{align*}
Using Taylor's formula, we find that, for any
$z\in\mathbb{C}$, $S^{(N)}_z\to e^{zb}$ as $N\to\infty$. Let
$\eta$ be the same as in Lemma \ref{analytic}. From
this and an argument similar to that used in the proof of Lemma
\ref{analytic}, we deduce that, for any $z\in\mathbb{C}$
satisfying $|z|\leq\eta$,
\begin{align}\label{20250106.2003}
\left\|\Phi_z^{(N)}(f)-\Phi_z(f)\right\|_{L^q(\omega^q)}\to0
\end{align}
as $N\to\infty$. Moreover, applying that $|S_z^{(N)}|\leq
e^{\eta|b|}$ for any $z\in\{z\in\mathbb{C}:\ |z|\leq\eta\}$ and
$N\in\mathbb{N}$, \eqref{20250106.1610}, the weighted
boundedness of $T$, and $\omega e^{2\eta|b|}
\boldsymbol{1}_{\mathrm{supp}(f)}\in L^r$ (see the proof
of Lemma \ref{analytic}), we obtain
\begin{align*}
\sup_{N\in\mathbb{N}}\sup_{\{z\in\mathbb{C}:\ |z|\leq A\}}
\left\|\Phi_z^{(N)}(f)\right\|_{L^q(\omega^q)}&\leq
\sup_{N\in\mathbb{N}}\sup_{\{z\in\mathbb{C}:\ |z|\leq A\}}
\left\|e^{\eta|b|}T\left(S_{-z}^{(N)}f\right)
\right\|_{L^q(\omega^q)}\\
&\lesssim\sup_{N\in\mathbb{N}}\sup_{\{z\in\mathbb{C}:\ |z|
\leq A\}}\left\|e^{\eta|b|}S_{-z}^{(N)}f
\right\|_{L^r(\omega^r)}\\
&\leq\left\|e^{2\eta|b|}f\right\|_{L^r(\omega^r)}\\
&\leq\|f\|_{L^\infty}\left[\int_{\mathrm{supp}(f)}
\omega^re^{2r\eta|b|}\right]^\frac{1}{r}<\infty.
\end{align*}
This, together with \cite[Proposition 1.2.2]{Hytonen-bookI},
further implies that, for any given $\eta\in(0,A)$ and
for any $N\in\mathbb{N}$,
\begin{align}\label{20250106.2035}
\left\|\frac{k!}{2\pi i}\int_{\{z\in\mathbb{C}:\ |z|=\eta\}}
\frac{\Phi_z^{(N)}(f)}{z^{k+1}}\,dz\right\|_{L^q(\omega^q)}
\lesssim\frac{k!}{\eta^k}\|f\|_{L^\infty}
\left[\int_{\mathrm{supp}(f)}\omega^re^{2r\eta|b|}\right]
^\frac{1}{r}<\infty.
\end{align}
Furthermore, by \eqref{20250106.2034} and \cite[Proposition
1.2.2]{Hytonen-bookI} again, we find that the Bochner integral
\begin{align*}
\frac{k!}{2\pi i}\int_{\{z\in\mathbb{C}:\ |z|=\eta\}}
\frac{\Phi_z(f)}{z^{k+1}}\,dz
\end{align*}
exists and determine a bounded operator from $C_\mathrm{c}\subset
L^r(\omega^r)$ to $L^q(\omega^q)$. To be more precise,
\begin{align}\label{20250106.2133}
\left\|\frac{k!}{2\pi i}\int_{\{z\in\mathbb{C}:\ |z|=\eta\}}
\frac{\Phi_z(f)}{z^{k+1}}\,dz\right\|_{L^q(\omega^q)}\lesssim
\frac{1}{\eta^k}\|f\|_{L^r(\omega^r)}.
\end{align}
From this, \eqref{20250106.2003}, \eqref{20250106.2035}, and the
dominated convergence theorem for the Bochner integral (see, for
example, \cite[Proposition 1.2.5]{Hytonen-bookI}), we infer that
\begin{align}\label{2255}
\left\|\frac{k!}{2\pi i}\int_{\{z\in\mathbb{C}:\ |z|=\eta\}}
\frac{\Phi_z^{(N)}(f)}{z^{k+1}}\,dz-\frac{k!}{2\pi i}
\int_{\{z\in\mathbb{C}:\ |z|=\eta\}}\frac{\Phi_z(f)}{z^{k+1}}\,dz
\right\|_{L^q(\omega^q)}\to0
\end{align}
as $N\to\infty$.

On the other hand, by the Cauchy integral formula, we conclude that,
for any $N\in\mathbb{N}\cap(k,\infty)$ and
$(x,t)\in\mathbb{R}^{n+1}$,
\begin{align*}
[b,T]_k(f)(x,t)&=T\left([b(x,t)-b(\cdot)]^{k}f(\cdot)\right)(x,t)\\
&=k!\sum_{j=0}^k\frac{[b(x,t)]^{k-j}}{(k-j)!}
T\left(\frac{[-b(\cdot)]^j}{j!}f(\cdot)\right)(x,t)\\
&=\sum_{j=0}^N\sum_{m=0}^N\frac{[b(x,t)]^m}{m!}
T\left(\frac{[-b(\cdot)]^j}{j!}f(\cdot)\right)(x,t)\\
&\quad\times
\frac{k!}{2\pi i}\int_{\{z\in\mathbb{C}:\ |z|=\eta\}}
z^{j+m-k-1}\,dz\\
&=\frac{k!}{2\pi i}\int_{\{z\in\mathbb{C}:\ |z|=\eta\}}
\frac{\Phi_z^{(N)}(f)(x,t)}{z^{k+1}}\,dz.
\end{align*}
Using this and \eqref{2255}, we find that
\begin{align*}
[b,T]_k(f)=\frac{k!}{2\pi i}
\int_{\{z\in\mathbb{C}:\ |z|=\eta\}}
\frac{\Phi_z(f)}{z^{k+1}}\,dz,
\end{align*}
which, together with \eqref{20250106.2133}, further implies that
\begin{align*}
\|[b,T]_k(f)\|_{L^q(\omega^q)}\lesssim
\frac{1}{\eta^k}\|f\|_{L^r(\omega^r)}.
\end{align*}
Therefore, \eqref{20250106.2136} holds, which
completes the proof of Theorem \ref{commutator boundedness}.
\end{proof}

\section{Parabolic Commutators and Their Applications to
Characterizing Parabolic BMO-Type Spaces}
\label{section3}

This section is divided into three subsections. In Subsection
\ref{subsection3.1}, we prove the mutual equivalences of
(i) through (vii) of Theorem \ref{commutator theorem 2}.
In Subsection \ref{subsection3.2}, we characterize the parabolic
BMO space and parabolic Campanato spaces in terms of commutators
of parabolic maximal operators with time lag. In Subsection
\ref{subsection3.3}, we show Theorem \ref{[I_alpha,b]} and that
(i) and (viii) through (xiii) of Theorem
\ref{commutator theorem 2} are mutually equivalent.

\subsection{Parabolic Fractional Maximal Commutators with Time Lag}
\label{subsection3.1}

The classical maximal commutator, introduced by Garc\'ia-Cuerva
et al. \cite{ghst(iumj-1991)}, plays a crucial role in
investigating the weighted boundedness of commutators of
singular integral operators. They also proved that the
boundedness of the maximal commutator can characterize the space
BMO. We refer to \cite{az(bkms-2023), st(iumj-1989),
st(tams-1993)} for more studies regarding the maximal commutator
and to \cite{lr(prse-2005), lr(mia-2007)} for more studies on its
one-sided version. The following is the definition of parabolic
fractional maximal commutators with time lag.

\begin{definition}\label{parabolic maximal commutator}
Let $\gamma,\alpha\in[0,1)$, $k\in\mathbb{N}$, and $b\in
L_{\mathrm{loc}}^1$. The \emph{$k$th order parabolic fractional
maximal commutators $M_{\alpha,b}^{\gamma+,k}$ and
$M_{\alpha,b}^{\gamma-,k}$ with time lag} are defined,
respectively, by setting, for any $f\in L_{\mathrm{loc}}^1$
and $(x,t)\in\mathbb{R}^{n+1}$,
\begin{align*}
M_{\alpha,b}^{\gamma+,k}(f)(x,t):=
\sup_{\genfrac{}{}{0pt}{}{R\in
\mathcal{R}_{p}^{n+1}}{(x,t)\in R^-(\gamma)}}
\left|R^+(\gamma)\right|^\alpha\fint_{R^+(\gamma)}
\left|b(x,t)-b(y,s)\right|^k|f(y,s)|\,dy\,ds
\end{align*}
and
\begin{align*}
M_{\alpha,b}^{\gamma-,k}(f)(x,t):=
\sup_{\genfrac{}{}{0pt}{}{R\in
\mathcal{R}_{p}^{n+1}}{(x,t)\in R^+(\gamma)}}
\left|R^-(\gamma)\right|^\alpha\fint_{R^-(\gamma)}
\left|b(x,t)-b(y,s)\right|^k|f(y,s)|\,dy\,ds.
\end{align*}
\end{definition}

Let $\beta\in(0,1]$. The \emph{parabolic Lipschitz space
$\mathrm{PLip}^\beta$} is defined to be the space of all
functions $f$ on $\mathbb{R}^{n+1}$ such that
\begin{align*}
\|f\|_{\mathrm{PLip}^\beta}:=
\sup_{\genfrac{}{}{0pt}{}{(x,t),(y,s)\in\mathbb{R}^{n+1}}
{(x,t)\neq(y,s)}}\frac{|f(x,t)-f(y,s)|}
{[d_p((x,t),(y,s))]^\beta}<\infty.
\end{align*}
We have the following coincidence of the parabolic Lipschitz
space and the parabolic Campanato space, which is a special case
of \cite[Theorem 5]{ms(am-1979)}.

\begin{lemma}\label{Campanato and Lip}
Let $\beta\in(0,\frac{1}{n+p})$. Then
$\mathrm{PC}^{\beta}$ and $\mathrm{PLip}^{(n+p)\beta}$
coincide with equivalent norms. Moreover,
for any $f\in L_{\mathrm{loc}}^1$,
\begin{align*}
\|f\|_{\mathrm{PC}^{\beta}}\sim
\|f\|_{\mathrm{PLip}^{(n+p)\beta}},
\end{align*}
where the positive equivalence constants are independent of $f$.
\end{lemma}

Now, applying Lemma \ref{Campanato and Lip}, we
are ready to show the mutual equivalence among (i) through (vii)
of Theorem \ref{commutator theorem 2}, that is, the Campanato
space can be characterized in terms of the weighted boundedness
of the parabolic fractional maximal commutators.

\begin{proof}[Proof of Mutual Equivalences among (i) through (vii)
of Theorem \ref{commutator theorem 2}]
We first prove (i)\ $\Longrightarrow$\ (ii), (iv), and (vi). Let
$k\in\mathbb{N}$ and $\gamma\in(0,1)$. From Lemma
\ref{Campanato and Lip}, we deduce
that, for any given $R\in\mathcal{R}_p^{n+1}$ and for any $f\in
L_{\mathrm{loc}}^1$ and $(x,t)\in R^-(\gamma)$,
\begin{align*}
&\left|R^+(\gamma)\right|^\alpha
\fint_{R^+(\gamma)}|b(x,t)-b(y,s)|^k|f(y,s)|\,dy\,ds\\
&\quad\leq\|b\|_{\mathrm{PLip}^{(n+p)\beta}}^k
\left|R^+(\gamma)\right|^\alpha
\fint_{R^+(\gamma)}\left[d_p((x,t),(y,s))\right]
^{(n+p)\beta k}|f(y,s)|\,dy\,ds\\
&\quad\lesssim\|b\|_{\mathrm{PC}^\beta}
^k\left|R^+(\gamma)\right|^{\alpha+\beta k}
\fint_{R^+(\gamma)}|f|\leq\|b\|_{\mathrm{PC}^\beta}^k
M_{\alpha+\beta k}^{\gamma+}(f)(x,t),
\end{align*}
which further implies that
\begin{align*}
M_{\alpha,b}^{\gamma+,k}(f)(x,t)\lesssim
\|b\|_{\mathrm{PC}^\beta}^k
M_{\alpha+\beta k}^{\gamma+}(f)(x,t).
\end{align*}
Combining this and Lemma \ref{weighted inequality uncentered}, we
obtain (ii), (iv), and (vi).

The proofs of (ii)\ $\Longrightarrow$\ (iii), (iv)\
$\Longrightarrow$\ (v), and (vi)\ $\Longrightarrow$\ (vii)
are trivial. Next, we show (iii) $\Longrightarrow$ (i).
Let $R\in\mathcal{R}_p^{n+1}$. Using Definition
\ref{parabolic maximal commutator},
we conclude that, for any $(x,t)\in R^-(\gamma)$,
\begin{align}\label{20250405.2200}
\left|R^+(\gamma)\right|^\alpha
\fint_{R^+(\gamma)}|b(x,t)-b(y,s)|\,dy\,ds\leq
M_{\alpha,b}^{\gamma+,1}\left(\boldsymbol{1}
_{R^+(\gamma)}\right)(x,t).
\end{align}
From this, H\"older's inequality, and the boundedness of
$M_{\alpha,b}^{\gamma+,1}$ from $L^r$ to $L^q$, we infer that
\begin{align}\label{20250305.2233}
&\frac{1}{|R^-(\gamma)|^\beta}\fint_{R^-(\gamma)}
\left|b-b_{R^-(\gamma)}\right|\\
&\quad\lesssim\frac{1}{|R^-(\gamma)|^\beta}\inf_{c\in\mathbb{R}}
\fint_{R^-(\gamma)}|b-c|
\leq\frac{1}{|R^-(\gamma)|^\beta}\fint_{R^-(\gamma)}
\left|b-b_{R^+(\gamma)}\right|\notag\\
&\quad\leq\frac{1}{|R^-(\gamma)|^\beta}\fint_{R^-(\gamma)}
\fint_{R^+(\gamma)}|b(x,t)-b(y,s)|\,dy\,ds\,dx\,dt\notag\\
&\quad\leq\frac{1}{|R^-(\gamma)|^{\beta+\alpha}}
\fint_{R^-(\gamma)}M_{\alpha,b}^{\gamma+,1}
\left(\boldsymbol{1}_{R^+(\gamma)}\right)\notag\\
&\quad\leq\frac{1}{|R^-(\gamma)|^{\beta+\alpha+\frac{1}{q}}}
\left\|M_{\alpha,b}^{\gamma+,1}
\left(\boldsymbol{1}_{R^+(\gamma)}\right)\right\|_{L^q}\notag\\
&\quad\leq\frac{\|M_{\alpha,b}^{\gamma+,1}\|_{\mathscr{L}(L^r,
L^q)}}{|R^-(\gamma)|^{\beta+\alpha+\frac{1}{q}}}
\left\|\boldsymbol{1}_{R^+(\gamma)}\right\|_{L^r}=
\left\|M_{\alpha,b}^{\gamma+,1}\right\|_{\mathscr{L}(L^r,
L^q)},\notag
\end{align}
Taking the supremum over all $R\in\mathcal{R}_p^{n+1}$, we
obtain $b\in\mathrm{PC}^\beta$ and hence (i) holds.

Now, we prove (v)\ $\Longrightarrow$\ (i). Let
$R\in\mathcal{R}_p^{n+1}$. Applying \eqref{20250405.2200} and
the boundedness of $M_{\alpha,b}^{\gamma+,1}$ from $L^r$ to
$L^{q,\infty}$, we conclude that, for any $\lambda\in(0,\infty)$,
\begin{align*}
\left|R^-(\gamma)\cap
\left\{\left|b-b_{R^+(\gamma)}\right|>\lambda\right\}\right|
&\leq\left|R^-(\gamma)\cap \left\{M_{\alpha,b}^{\gamma+,1}
\left(\boldsymbol{1}_{R^+(\gamma)}\right)>\lambda
\left|R^+(\gamma)\right|^\alpha\right\}\right|\\
&\leq\frac{\|M_{\alpha,b}^{\gamma+,1}\|_{\mathscr{L}(L^r,
L^{q,\infty})}}{\lambda^q|R^+(\gamma)|^{\alpha q}}
\left\|\boldsymbol{1}_{R^+(\gamma)}\right\|_{L^r}^q\\
&=\frac{\|M_{\alpha,b}^{\gamma+,1}\|_{\mathscr{L}(L^r,
L^{q,\infty})}}{\lambda^q}\left|R^+(\gamma)\right|
^{\frac{q}{r}-\alpha q}.
\end{align*}
This, together with \eqref{20250305.2233}, Cavalieri's
principle (see, for example, \cite[(2.1)]{duo-book}), and
$\frac{1}{r}-\frac{1}{q}=\alpha+\beta$, further implies that
\begin{align*}
&\frac{1}{|R^-(\gamma)|^\beta}\fint_{R^-(\gamma)}
\left|b-b_{R^-(\gamma)}\right|\\
&\quad\lesssim\frac{1}{|R^-(\gamma)|^\beta}\fint_{R^-(\gamma)}
\left|b-b_{R^+(\gamma)}\right|\\
&\quad=\frac{1}{|R^-(\gamma)|^{\beta+1}}
\int_0^\infty\left|R^-(\gamma)\cap
\left\{\left|b-b_{R^+(\gamma)}\right|>\lambda
\right\}\right|\,d\lambda\\
&\quad=\frac{1}{|R^-(\gamma)|^{\beta+1}}
\left(\int_0^\delta+\int_\delta^\infty\right)\left|R^-(\gamma)\cap
\left\{\left|b-b_{R^+(\gamma)}\right|>\lambda
\right\}\right|\,d\lambda\\
&\quad\lesssim\left\|M_{\alpha,b}^{\gamma+,1}\right\|
_{\mathscr{L}(L^r\,L^{q,\infty})}
\left[\delta\left|R^-(\gamma)\right|
+\left|R^+(\gamma)\right|^{\frac{q}{r}-\alpha
q}\int_\delta^\infty\frac{1}{\lambda^q}
\,d\lambda\right]\sim\left\|M_{\alpha,b}^{\gamma+,1}\right\|
_{\mathscr{L}(L^r,L^{q,\infty})},
\end{align*}
where $\delta:=|R^-(\gamma)|^{\beta}$. Taking the supremum over
all $R\in\mathcal{R}_p^{n+1}$, we find that
$b\in\mathrm{PC}^\beta$ and hence (i) holds.

Finally, we show (vii)\ $\Longrightarrow$\ (i). Let
$R\in\mathcal{R}_p^{n+1}$. From \eqref{20250305.2233} with
$\gamma$ therein replaced by $\rho$, the boundedness of
$M_{\alpha,b}^{\rho+,1}$ from $L^r$ to $L^\infty$,
and $\frac{1}{r}=\alpha+\beta$, we deduce that
\begin{align*}
\frac{1}{|R^-(\rho)|^\beta}\fint_{R^-(\rho)}
\left|b-b_{R^-(\rho)}\right|
&\leq\frac{1}{|R^-(\rho)|^{\beta+\alpha}}
\fint_{R^-(\rho)}M_{\alpha,b}^{\rho+,1}
\left(\boldsymbol{1}_{R^+(\rho)}\right)\\
&\leq\frac{1}{|R^-(\rho)|^{\beta+\alpha}}
\left\|M_{\alpha,b}^{\rho+,1}
\left(\boldsymbol{1}_{R^+(\rho)}\right)\right\|_{L^\infty}\\
&\leq\frac{\|M_{\alpha,b}^{\rho+,1}\|_{\mathscr{L}(L^r,
L^\infty)}}{|R^-(\rho)|^{\beta+\alpha}}
\left\|\boldsymbol{1}_{R^+(\rho)}\right\|_{L^r}=
\left\|M_{\alpha,b}^{\rho+,1}\right\|_{\mathscr{L}(L^r,
L^\infty)}.
\end{align*}
Taking the supremum over all $R\in\mathcal{R}_p^{n+1}$, we
conclude that $b\in\mathrm{PC}^\beta$ and hence (i) holds.
This finishes the proof of the mutual equivalences of (i)
through (vii) of Theorem \ref{commutator theorem 2}.
\end{proof}

\subsection{Commutators of Parabolic Maximal Operators with Time Lag}
\label{subsection3.2}

The boundedness of the commutator of the Hardy--Littlewood
maximal operator was first obtained by Milman and Schonbek
\cite{ms(pams-1990)}; see also \cite{az(bkms-2023),
bmr(pams-2000)}. In this subsection, we establish the
characterization of the parabolic BMO space and parabolic
Campanato spaces in terms of the boundedness of commutators of
parabolic maximal operators with time lag. Recall that an
operator $T$ defined on a linear subspace $D(T)$ of
$L_{\mathrm{loc}}^1$ and taking
values in $\mathscr{M}$ is called a \emph{positive quasilinear
operator} if $T$ satisfies
\begin{enumerate}
\item[\rm(i)] $T(f)(x,t)\in[0,\infty)$ for any $f\in D(T)$ and
$(x,t)\in\mathbb{R}^{n+1}$;

\item[\rm(ii)] $T(\alpha f)=|\alpha|T(f)$ for any $f\in D(T)$
and $\alpha\in\mathbb{R}$;

\item[\rm(iii)] $|T(f)-T(g)|\leq T(f-g)$ for any $f,g\in D(T)$.
\end{enumerate}
For example, for any $\gamma,\alpha\in[0,1)$, both
$M_\alpha^{\gamma+}$ and $M_\alpha^{\gamma-}$ are positive
quasilinear operators. Let $T$ be a positive quasilinear
operator and $b\in L_{\mathrm{loc}}^1$. Recall that the
\emph{commutator $[b,T]$ of $T$ with $b$} is defined by setting,
for any suitable function $f$ on $\mathbb{R}^{n+1}$ and any
$(x,t)\in\mathbb{R}^{n+1}$,
\begin{align*}
[b,T](f)(x,t):=b(x,t)T(f)(x,t)-T(bf)(x,t).
\end{align*}
We simply denote $[b,M^{\gamma+}_0]$ and $[b,M^{\gamma-}_0]$,
respectively, by $[b,M^{\gamma+}]$ and $[b,M^{\gamma-}]$. The
commutator of the parabolic fractional maximal operator with
time lag and the parabolic fractional maximal commutator with
time lag are essentially different because the latter is both
positive and sublinear, while the former is neither. However,
they have the following relation.

\begin{lemma}\label{relationship between two commutators}
Let $\gamma,\alpha\in[0,1)$ and $b\in
L_{\mathrm{loc}}^1$. Then, for any
$f\in L_{\mathrm{loc}}^1$,
\begin{align*}
\left[b,M^{\gamma+}_\alpha\right](f)\leq
M_{\alpha,b}^{\gamma+,1}(f)+2b_-M^{\gamma+}_\alpha(f).
\end{align*}
\end{lemma}

\begin{proof}
Let $f\in L_{\mathrm{loc}}^1$. Since
\begin{align*}
\left|\left[b,M^{\gamma+}_\alpha\right](f)-
\left[|b|,M^{\gamma+}_\alpha\right](f)\right|=
2b_-M^{\gamma+}_\alpha(f),
\end{align*}
it follows that
\begin{align*}
\left|\left[b,M^{\gamma+}_\alpha\right](f)\right|&\leq
\left|\left[|b|,M^{\gamma+}_\alpha\right](f)\right|+
2b_-M^{\gamma+}_\alpha(f)\\
&\leq M_{\alpha,|b|}^{\gamma+,1}(f)+2b_-M^{\gamma+}_\alpha(f)
\leq M_{\alpha,b}^{\gamma+,1}(f)+2b_-M^{\gamma+}_\alpha(f).
\end{align*}
This finishes the proof of Lemma
\ref{relationship between two commutators}.
\end{proof}

The following is the main result of this subsection.

\begin{theorem}\label{commutator theorem 1}
Let $\gamma\in(0,1)$, $\beta\in(0,\frac{1}{n+p})$, and $b\in
L_{\mathrm{loc}}^1$. Then the following statements hold.
\begin{enumerate}
\item[\rm(i)] $b\in\mathrm{PBMO}$ and $b_-\in L^\infty$ if
and only if, for any (or some) $q\in(1,\infty)$, both
$[b,M^{\gamma+}]$ and $[b,M^{\gamma-}]$ are bounded on $L^q$.

\item[\rm(ii)] $b\in\mathrm{PC}^\beta$ and $b$ is
nonnegative if and only if, for any (or some) $1<r<q<\infty$ with
$\frac{1}{r}-\frac{1}{q}=\beta$, both $[b,M^{\gamma+}]$ and
$[b,M^{\gamma-}]$ are bounded from $L^r$ to $L^q$.
\end{enumerate}
\end{theorem}

Let $q\in(1,\infty)$. The \emph{Muckenhoupt class $A_q$} (based
on parabolic rectangles) is defined to be the set of all
weights $\omega$ such that
\begin{align*}
[\omega]_{A_q}:=\sup_{R\in\mathcal{R}_p^{n+1}}\fint_R\omega
\left(\fint_R\omega^{\frac{1}{1-q}}\right)^{q-1}<\infty.
\end{align*}
To prove Theorem \ref{commutator theorem 1},
we still need the following lemma for commutators of
positive quasilinear operators, which is precisely
\cite[Proposition 3]{bmr(pams-2000)}.

\begin{lemma}\label{commutator for positive quasilinear operators}
Let $\gamma\in(0,1)$, $q,r\in(1,\infty)$, $b\in
\mathrm{PBMO}$ with $b_-\in L^\infty$, and $T$ be a
positive quasilinear operator. If $T$ is bounded on
$L^q(\omega)$ for any $\omega\in A_r$, then $[b,T]$
is bounded on $L^q$ and the operator norm of $[b,T]$ is
less than a constant multiple of
$\|b\|_{\mathrm{PBMO}}+\|b_-\|_{L^\infty}$.
\end{lemma}

We are ready to prove Theorem \ref{commutator theorem 1}.

\begin{proof}[Proof of Theorem \ref{commutator theorem 1}]
We first show the necessity of (i). Fix $q\in(1,\infty)$.
Notice that both $M^{\gamma+}$ and $M^{\gamma-}$ are positive
quasilinear operator on $L_{\mathrm{loc}}^1$ and, using Lemma
\ref{weighted inequality uncentered}(i) and the obvious fact that
$A_q\subset A_q^+(\gamma)\cap A_q^-(\gamma)$, we conclude that
both $M^{\gamma+}$ and $M^{\gamma-}$ are bounded on
$L^q(\omega)$ for any $\omega\in A_q$. From this and Lemma
\ref{commutator for positive quasilinear operators},
we infer that both $[b,M^{\gamma+}]$ and $[b,M^{\gamma-}]$
are bounded on $L^q$, which completes the
proof of the necessity.

Next, we prove the sufficiency of (i). To do this, we first
show that $b_-\in L^\infty$. For any given
$R_0\in\mathcal{R}_p^{n+1}$, define the \emph{uncentered
parabolic maximal operators $M^{\gamma+}_{R_0}$
and $M^{\gamma-}_{R_0}$ related to $R_0$ with time lag},
respectively, by setting, for any
$f\in L_{\mathrm{loc}}^1$ and $(x,t)\in R_0$,
\begin{align*}
M^{\gamma+}_{R_0}(f)(x,t):=\sup_{\genfrac{}{}{0pt}{}{R\in
\mathcal{R}_{p}^{n+1},\,R\subset R_0}{(x,t)\in
R^-(\gamma)}}\fint_{R^+(\gamma)}|f|
\end{align*}
and
\begin{align*}
M^{\gamma-}_{R_0}(f)(x,t):=\sup_{\genfrac{}{}{0pt}{}{R\in
\mathcal{R}_{p}^{n+1},\,R\subset R_0}{(x,t)\in
R^+(\gamma)}}\fint_{R^-(\gamma)}|f|.
\end{align*}
For any given $R\in\mathcal{R}_p^{n+1}$,
let $f:=\boldsymbol{1}_{R^+(\gamma)}$. Then $f\in L^q$. By a
simple computation, we obtain $M^{\gamma+}(f)(x,t)=1$ and
$M^{\gamma+}(bf)(x,t)=M^{\gamma+}_R(b)(x,t)$ for any $(x,t)\in
R^-(\gamma)$. This, together with that $[b,M^{\gamma+}]$ is
bounded on $L^q$, further implies that
\begin{align*}
\left[\int_{R^-(\gamma)}\left|b-M^{\gamma+}_R(b)\right|^q
\right]^{\frac{1}{q}}&=\left[\int_{R^-(\gamma)}\left|b
M^{\gamma+}(f)-M^{\gamma+}(bf)\right|^q\right]^{\frac{1}{q}}\\
&\leq\left\|\left[b,M^{\gamma+}\right](f)\right\|
_{L^q}\leq\left\|\left[b,M^{\gamma+}\right]\right\|
_{\mathscr{L}(L^q)}\|f\|_{L^q}\\
&=\left\|\left[b,M^{\gamma+}\right]\right\|
_{\mathscr{L}(L^q)}\left|R^-(\gamma)\right|^{\frac{1}{q}}.
\end{align*}
From this and H\"older's inequality, it follows that
\begin{align}\label{20240624.1516}
\fint_{R^-(\gamma)}\left|b-M^{\gamma+}_R(b)\right|\leq
\left[\fint_{R^-(\gamma)}\left|b-M^{\gamma+}_R(b)\right|^q
\right]^{\frac{1}{q}}\leq
\left\|\left[b,M^{\gamma+}\right]\right\|_{\mathscr{L}(L^q)}.
\end{align}
In addition, applying Lebesgue's differentiation
theorem (see \cite[Lemma 2.3]{kmy(ma-2023)}), we conclude that,
for any given $R\in\mathcal{R}_p^{n+1}$, $|b|\leq
M^{\gamma+}_R(b)$ almost everywhere in $R^-(\gamma)$. Thus,
\begin{align}\label{20241122.1349}
0\leq b_-\leq M^{\gamma+}_R(b)-b_++b_-=M^{\gamma+}_R(b)-b
\end{align}
almost everywhere in $R^-(\gamma)$. From this and
\eqref{20240624.1516}, we infer that $(b)_{R^-(\gamma)}
\leq\|[b,M^{\gamma+}]\|_{\mathscr{L}(L^q)}$, which, together
with Lebesgue's differentiation theorem and the arbitrariness of
$R$, further implies $b_-\in L^\infty$.

Now, we prove $b\in\mathrm{PBMO}$.
Similar to \eqref{20240624.1516}, we can
show that, for any $R\in\mathcal{R}_p^{n+1}$,
\begin{align*}
\fint_{R^+(\gamma)} \left|b-M^{\gamma-}_R(b)\right|\leq
\left[\fint_{R^+(\gamma)} \left|b-M^{\gamma-}_R(b)\right|^q
\right]^{\frac{1}{q}}\leq
\left\|\left[b,M^{\gamma-}\right]\right\|_{\mathscr{L}(L^q)},
\end{align*}
which, together with the fact that $|b|\leq M^{\gamma-}_R(b)$
almost everywhere in $R^+(\gamma)$, further implies that
\begin{align*}
b_{R^-(\gamma)}\leq|b|_{R^-(\gamma)}\leq
\fint_{R^+(\gamma)}M^{\gamma+}_R(b)\leq b_{R^+(\gamma)}+
\left\|\left[b,M^{\gamma-}\right]\right\|_{\mathscr{L}(L^q)}.
\end{align*}
Combining this and \eqref{20240624.1516}, we find that,
for any given $R\in\mathcal{R}_p^{n+1}$,
\begin{align}\label{20241122.1357}
\fint_{R^-(\gamma)}\left|b-b_{R^+(\gamma)}\right|&=
\frac{1}{|R^-(\gamma)|}\int_{R^-(\gamma)}
\left[b-b_{R^+(\gamma)}\right]\\
&\quad+\frac{1}{|R^-(\gamma)|}\int_{R^-(\gamma)\cap\{b<
b_{R^+(\gamma)}\}}\left[b_{R^+(\gamma)}-b\right]\notag\\
&\quad-\frac{1}{|R^-(\gamma)|}\int_{R^-(\gamma)\cap
\{b<b_{R^+(\gamma)}\}}\left[b-b_{R^+(\gamma)}\right]\notag\\
&=b_{R^-(\gamma)}-b_{R^+(\gamma)}+\frac{2}{|R^-(\gamma)|}
\int_{R^-(\gamma)\cap\{b<b_{R^+(\gamma)}\}}
\left[b_{R^+(\gamma)}-b\right]\notag\\
&\leq\left\|\left[b,M^{\gamma-}\right]\right\|_{\mathscr{L}(L^q)}
+2\fint_{R^-(\gamma)}\left|b-M^{\gamma+}_R(b)\right|\notag\\
&\lesssim\left\|\left[b,M^{\gamma+}\right]\right\|
_{\mathscr{L}(L^q)}+\left\|\left[b,M^{\gamma-}\right]\right\|
_{\mathscr{L}(L^q)},\notag
\end{align}
which further implies that $b\in\mathrm{PBMO}$, which
completes the proof of the sufficiency of (i).

Next, we prove (ii). From Lemma
\ref{relationship between two commutators} and (i)\
$\Longrightarrow$\ (ii) of Theorem \ref{commutator theorem 2}, we
deduce that the necessity holds. Then we show the sufficiency.
Let $R\in\mathcal{R}_p^{n+1}$. Similar to \eqref{20240624.1516},
we can prove that
\begin{align*}
\fint_{R^-(\gamma)}\left|b-M^{\gamma+}_R(b)\right|\leq
\left[\fint_{R^-(\gamma)}\left|b-M^{\gamma+}_R(b)\right|^q
\right]^{\frac{1}{q}}\leq
\left\|\left[b,M^{\gamma+}\right]\right\|
_{\mathscr{L}(L^q)}\left|R^-(\gamma)\right|^\beta.
\end{align*}
By this and \eqref{20241122.1349}, we obtain
\begin{align*}
\fint_{R^-(\gamma)}b_-\leq
\fint_{R^-(\gamma)}\left|b-M^{\gamma+}_R(b)\right|\leq
\left\|\left[b,M^{\gamma+}\right]\right\|
_{\mathscr{L}(L^q)}\left|R^-(\gamma)\right|^\beta,
\end{align*}
which, together with Lebesgue's differentiation theorem,
further implies that $b_-=0$ almost everywhere on
$\mathbb{R}^{n+1}$ and hence $b$ is nonnegative. In addition,
By a slight modification of the estimation of
\eqref{20241122.1357}, we find that
\begin{align*}
\frac{1}{|R^-(\gamma)|^\beta}\fint_{R^-(\gamma)}
\left|b-b_{R^+(\gamma)}\right|\lesssim
\left\|\left[b,M^{\gamma+}\right]\right\|_{\mathscr{L}(L^q)}+
\left\|\left[b,M^{\gamma-}\right]\right\|_{\mathscr{L}(L^q)},
\end{align*}
and hence $b\in\mathrm{PC}^\beta$. This finishes the
proof of the sufficiency of (ii) and hence Theorem
\ref{commutator theorem 1}.
\end{proof}

\subsection{Commutators of Parabolic Fractional Integral Operators with Time Lag}
\label{subsection3.3}

The boundedness of the commutator of the classical fractional
integral operators was first studied by Chanillo
\cite{c(iumj-1982)}. We refer to \cite{bmmst(ma-2020),
cm(pm-2012), st(pm-1991)} for more investigations into it and to
\cite{bl(ieot-2008), fhll(jga-2023), lr(prse-2005),
lr(mia-2007)} for more studies on the commutator of the
one-sided fractional integral operators. In the parabolic
setting, the \emph{parabolic fractional integral operator
$I_\alpha^{\gamma+}$ with time lag},
originally introduced in \cite[Definition 5.1]{kyyz-2024-2},
is defined by setting, for
any $f\in L_{\mathrm{loc}}^1$ and $(x,t)\in\mathbb{R}^{n+1}$,
\begin{align}\label{20250304.2134}
I_\alpha^{\gamma+}(f)(x,t):=\int_{\bigcup_{L\in(0,\infty)}R
(x,t,L)^+(\gamma)}\frac{f(y,s)}
{[d_p((x,t),(y,s))]^{(n+p)(1-\alpha)}}\,dy\,ds,
\end{align}
where $\gamma,\alpha\in(0,1)$,
and $I_\alpha^{\gamma-}$ can be similarly defined.
It is known that the parabolic weighted boundedness of the
parabolic fractional integral operator with time lag can
characterize $A_{r,q}^+(\gamma)$ weights; see
\cite[Corollary 5.6 and Theorem 5.7]{kyyz-2024-2}.

\begin{lemma}\label{weighted inequality fractional integral}
Let $\gamma,\alpha\in(0,1)$ and $\omega$ be a weight.
Then the following assertions hold.
\begin{enumerate}
\item[\rm(i)] Let $1<r<q<\infty$ with
$\frac{1}{r}-\frac{1}{q}=\alpha$. Then $\omega\in
A_{r,q}^+(\gamma)$ if and only if $I^{\gamma+}_\alpha$ is
bounded from $L^r(\omega^r)$ to
$L^q(\omega^q)$.

\item[\rm(ii)] Let $1\leq r<q<\infty$
with $\frac{1}{r}-\frac{1}{q}=\alpha$. Then $\omega\in
A_{r,q}^+(\gamma)$ if and only if $I^{\gamma+}_\alpha$ is
bounded from $L^r(\omega^r)$ to $L^{q,\infty}(\omega^q)$.
\end{enumerate}
\end{lemma}

Using Theorem \ref{parabolic extrapolation for commutator} and
integrating the Fourier series expansion argument in
\cite{j(am-1978)} and the parabolic geometry associated with the
integral domain in \eqref{20250304.2134}, we can
show Theorem \ref{[I_alpha,b]}.

\begin{proof}[Proof of Theorem \ref{[I_alpha,b]}]
From Theorem \ref{parabolic extrapolation for commutator} and
Lemma \ref{weighted inequality fractional integral}(i),
we deduce that, for any given $k\in\mathbb{N}$, $1<r<q<\infty$
satisfying $\frac{1}{r}-\frac{1}{q}=\alpha$, and $\omega\in
A_{r,q}^+(\gamma)$, and for any $f\in L^r(\omega^r)$,
\begin{align*}
\left\|\left[b,I_\alpha^{\gamma+}\right]_k(f)\right\|_{L^q(\omega^q)}
\lesssim\|b\|_{\mathrm{PBMO}}^k\|f\|_{L^r(\omega^r)}.
\end{align*}
This finishes the proof of (i)\ $\Longrightarrow$\ (ii).
The proof of (ii)\ $\Longrightarrow$\ (iii) is trivial. We
prove (iii) $\Longrightarrow$ (i). Indeed, let
$\delta:=\frac{1}{8(\sqrt{n}+1)}$
and $(z_0,\tau_0):=(1,\dots,1,-1-2\delta-2\gamma\delta
(2+\frac{1}{\delta})^\frac{1}{p})\in\mathbb{R}^{n+1}$.
Then, for any $(x,t)\in\mathbb{R}^{n+1}$ satisfying
$|x-z_0|+|t-\tau_0|^\frac{1}{p}\in[0,\frac{1}{2})$,
the function $(|x|+|t|^\frac{1}{p})^{(n+p)(1-\alpha)}$ can be
expressed as an absolutely convergent Fourier series, that is,
\begin{align}\label{20250109.1630}
\left(|x|+|t|^\frac{1}{p}\right)^{(n+p)(1-\alpha)}
=\sum_{m\in\mathbb{N}}a_me^{i\langle v_m,(x,t)\rangle},
\end{align}
where, for any $m\in\mathbb{N}$, $a_m\in\mathbb{C}$ and
$v_m\in\mathbb{R}^{n+1}$ and the symbol
$\langle\cdot,\cdot\rangle$ represents the standard inner product
on $\mathbb{R}^{n+1}$. For any
$R:=R(x_0,t_0,L_0)\in\mathcal{R}_p^{n+1}$ with
$(x_0,t_0)\in\mathbb{R}^{n+1}$ and $L_0\in(0,\infty)$, let
$y_0:=x_0-\frac{L_0z_0}{\delta}$,
$s_0:=t_0-\frac{L_0^p\tau_0}{\delta^p}$,
and $R':=R(y_0,s_0,L_0)$. Through some simple calculations, we
conclude that
\begin{enumerate}
\item[\rm(i)] for any $(x,t)\in R$ and $(y,s)\in R'$,
$|\frac{x-y}{L_0}-\frac{z_0}{\delta}|+
|\frac{t-s}{L_0^p}-\frac{\tau_0}{\delta^p}|^\frac{1}{p}
\leq2(\sqrt{n}+1)=\frac{1}{4\delta}$;

\item[(ii)] for any $(x,t)\in R$, $R'\subset
\bigcup_{L\in(0,\infty)}R(x,t,L)^+(\gamma)$.
\end{enumerate}
For any $(x,t)\in\mathbb{R}^{n+1}$, let
$S(x,t):=\mathrm{sgn}[b(x,t)-b_{R'}]$.
By this, (i), (ii), \eqref{20250109.1630}, and Fubini's
theorem, we find that
\begin{align}\label{20250109.1640}
\int_R|b-b_{R'}|&=\int_R(b-b_{R'})S\\
&=\frac{1}{|R'|}\int_R\int_{R'}\left[b(x,t)-b(y,s)\right]S(x,t)
\,dy\,ds\,dx\,dt\notag\\
&\sim|R|^{-\alpha}\int_{\mathbb{R}^{n+1}}
\int_{\bigcup_{L\in(0,\infty)}R(x,t,L)^+(\gamma)}
\frac{b(x,t)-b(y,s)}{[d_p((x,t),(y,s))]
^{(n+p)(1-\alpha)}}\notag\\
&\quad\times S(x,t)\left[\left|\frac{\delta(x-y)}{L_0}\right|
+\left|\frac{\delta^p(t-s)}{L_0^p}\right|^\frac{1}{p}\right]
^{(n+p)(1-\alpha)}\notag\\
&\quad\times\boldsymbol{1}_R(x,t)\boldsymbol{1}_{R'}(y,s)
\,dy\,ds\,dx\,dt\notag\\
&=|R|^{-\alpha}\sum_{m\in\mathbb{N}}a_m\int_{\mathbb{R}^{n+1}}
\int_{\bigcup_{L\in(0,\infty)}R(x,t,L)^+(\gamma)}
\frac{b(x,t)-b(y,s)}{[d_p((x,t),(y,s))]
^{(n+p)(1-\alpha)}}\notag\\
&\quad\times S(x,t)e^{i\langle v_m,(\frac{\delta(x-y)}{L_0},
\frac{\delta(t-s)}{L_0^p})\rangle}\boldsymbol{1}_R(x,t)
\boldsymbol{1}_{R'}(y,s)\,dy\,ds\,dx\,dt.\notag
\end{align}
For any $m\in\mathbb{N}$ and $(x,t)\in\mathbb{R}^{n+1}$, define
\begin{align*}
f_m(x,t):=e^{-i\langle v_m,(\frac{\delta x}{L_0},
\frac{\delta t}{L_0^p})\rangle}\boldsymbol{1}_{R'}(x,t)
\mbox{\ \ and\ \ }g_m(x,t):=e^{i\langle v_m,(\frac{\delta x}{L_0},
\frac{\delta t}{L_0^p})\rangle}\boldsymbol{1}_R(x,t).
\end{align*}
Then, for any $m\in\mathbb{N}$, $f_m\in L^r$ and
$\|f_m\|_{L^r}\leq |R'|^\frac{1}{r}$. From this,
\eqref{20250109.1640}, H\"older's inequality, the boundedness of
$[b,I_\alpha^{\gamma+}]$ from $L^r$ to $L^q$, and
$\sum_{m\in\mathbb{N}}|a_m|<\infty$, we infer that
\begin{align*}
\int_R\left|b-b_{R'}\right|&\lesssim
|R|^{-\alpha}\sum_{m\in\mathbb{N}}|a_m|\int_{\mathbb{R}^{n+1}}
\left|\left[b,I_\alpha^{\gamma+}\right](f_m)\right|\,|g_m|\\
&\leq|R|^{-\alpha}\sum_{m\in\mathbb{N}}|a_m|\int_R
\left|\left[b,I_\alpha^{\gamma+}\right](f_m)\right|\\
&\leq|R|^{-\alpha}\sum_{m\in\mathbb{N}}|a_m|\,|R|^\frac{1}{q'}
\left[\int_{\mathbb{R}^{n+1}}\left|
\left[b,I_\alpha^{\gamma+}\right](f_m)\right|^q\right]
^\frac{1}{q}\\
&\lesssim|R|^{-\alpha}|R|^\frac{1}{q'}
\sum_{m\in\mathbb{N}}|a_m|\,
\left\|\left[b,I_\alpha^{\gamma+}\right]\right\|
_{\mathscr{L}(L^r,L^q)}
\|f\|_{L^r}\lesssim
\left\|\left[b,I_\alpha^{\gamma+}\right]\right\|
_{\mathscr{L}(L^r,L^q)}|R|.
\end{align*}
Taking the supremum over all $R\in\mathcal{R}_p^{n+1}$,
we conclude that
\begin{align*}
\|b\|_{\mathrm{PBMO}}\sim\sup_{R\in\mathcal{R}_p^{n+1}}
\inf_{c\in\mathbb{R}}\fint_R|b-c|\leq
\sup_{R\in\mathcal{R}_p^{n+1}}\fint_R|b-b_{R'}|\lesssim
\left\|\left[b,I_\alpha^{\gamma+}\right]\right\|
_{\mathscr{L}(L^r,L^q)}
\end{align*}
and hence $b\in\mathrm{PBMO}$. This finishes the proof of
(iii)\ $\Longrightarrow$\ (i) and Theorem \ref{[I_alpha,b]}.
\end{proof}

The following corollary is an immediate consequence of
Theorem \ref{[I_alpha,b]}.

\begin{corollary}\label{[I_alpha,b] cor}
Let $\gamma,\alpha\in(0,1)$, $1<r<q<\infty$ satisfy
$\frac{1}{r}-\frac{1}{q}=\alpha$, and $b\in L_{\mathrm{loc}}^1$.
Then the following statements are mutually equivalent.
\begin{enumerate}
\item[\rm(i)] $b\in\mathrm{PBMO}$.

\item[\rm(ii)] For any even positive integer $k$ and any
$\omega\in A_{r,q}^+(\gamma)$, $M_{\alpha,b}^{\gamma+,k}$ is
bounded from $L^r(\omega^r)$ to $L^q(\omega^q)$.

\item[\rm(iii)] There exists $k_0\in\mathbb{N}$ such that
$M_{\alpha,b}^{\gamma+,k_0}$ is bounded from $L^r$ to $L^q$.
\end{enumerate}
\end{corollary}

\begin{proof}
The proof of (i)\ $\Longrightarrow$\ (ii) follows immediately
from Theorem \ref{[I_alpha,b]}, \cite[Theorem 3.1]{km(am-2024)},
and the fact that, for any $f\in
L_{\mathrm{loc}}^1$,
\begin{align}\label{20250119.1120}
M_\alpha^{\gamma+}(f)\lesssim
I_\alpha^{\frac{\gamma+1}{2^p}+}(|f|).
\end{align}
The proof of (ii)\ $\Longrightarrow$\ (iii) is trivial and,
applying an argument similar to that used in the
proof of (iii)\ $\Longrightarrow$\ (i) in Theorem
\ref{commutator theorem 2} with $\beta=0$ therein, we obtain
(iii)\ $\Longrightarrow$\ (i), which completes the proof of
Corollary \ref{[I_alpha,b] cor}.
\end{proof}

\begin{remark}
\begin{enumerate}
\item[\rm(i)] It is unknown whether (i)\ $\Longrightarrow$\ (ii)
holds for any odd positive integer $k$ in Corollary
\ref{[I_alpha,b] cor} because, in this case, the estimate
$M_{\alpha,b}^{\gamma+,k}(f)\lesssim
I_{\alpha,b}^{\gamma+,k}(f)$ may not hold for any
$f\in L_{\mathrm{loc}}^1$ and hence we can not use the weighted
boundedness of $I_{\alpha,b}^{\gamma+,k}$ to deduce that of
$M_{\alpha,b}^{\gamma+,k}$.

\item[\rm(ii)] Let $\gamma\in(0,1)$, $k\in\mathbb{N}$,
$q\in(1,\infty)$, and $b\in L_{\mathrm{loc}}^1$. It can be shown
in the same way as in the proof of Corollary \ref{[I_alpha,b] cor}
that the boundedness of
$M_{0,b}^{\gamma+,k}$ on $L^q$ implies $b\in\mathrm{PBMO}$.
However, it is still unknown whether
$b\in\mathrm{PBMO}$ implies the weighted boundedness of
$M_{0,b}^{\gamma+,k}$.
\end{enumerate}
\end{remark}

We now introduce the following concept of maximal commutators of
parabolic fractional integral operators with time lag.
\begin{definition}\label{maximal commutators of fractional integral}
Let $\gamma,\alpha\in[0,1)$, $k\in\mathbb{N}$, and $b\in
L_{\mathrm{loc}}^1$. The \emph{$k$th order maximal commutators
$I_{\alpha,b}^{\gamma+,k}$ and $I_{\alpha,b}^{\gamma-,k}$ of
parabolic fractional integral operators with time lag} are
defined, respectively, by setting, for any $f\in
L_{\mathrm{loc}}^1$ and $(x,t)\in\mathbb{R}^{n+1}$,
\begin{align*}
I_{\alpha,b}^{\gamma+,k}(f)(x,t):=
\int_{\bigcup_{L\in(0,\infty)}R
(x,t,L)^+(\gamma)}\frac{|b(x,t)-b(y,s)|^k|f(y,s)|}
{[d_p((x,t),(y,s))]^{(n+p)(1-\alpha)}}\,dy\,ds
\end{align*}
and
\begin{align*}
I_{\alpha,b}^{\gamma-,k}(f)(x,t):=
\int_{\bigcup_{L\in(0,\infty)}R
(x,t,L)^-(\gamma)}\frac{|b(x,t)-b(y,s)|^k|f(y,s)|}
{[d_p((x,t),(y,s))]^{(n+p)(1-\alpha)}}\,dy\,ds.
\end{align*}
\end{definition}

Finally, we prove that (i) and (viii) through (xiii) of
Theorem \ref{commutator theorem 2} are mutually equivalent,
which completes the proof of Theorem \ref{commutator theorem 2}.

\begin{proof}[Proof of Mutual Equivalences of (i) and (viii)
through (xiii) of Theorem \ref{commutator theorem 2}]
We first show that (i)\ $\Longrightarrow$\ (viii), (x), and (xii).
Assume that (i) holds. By Lemma \ref{Campanato and Lip}, we find
that, for any $f\in L_{\mathrm{loc}}^1$,
\begin{align*}
\max\left\{\left|\left[b,I_\sigma^{\gamma+}\right]_k(f)\right|,\,
I_{\sigma,b}^{\gamma+,k}(f)\right\}\lesssim
\|b\|_{\mathrm{PC}^\beta}^kI^{\gamma+}_{\sigma+\beta k}(|f|),
\end{align*}
which, together with Lemma
\ref{weighted inequality fractional integral}, further implies
that (viii), (x), and (xii) hold.

The proofs of (viii)\ $\Longrightarrow$\ (ix), (x)\
$\Longrightarrow$\ (xi), and (xii) $\Longrightarrow$ (xii) are
trivial. Assume that (ix) holds. From an argument similar to
that used in the proof of (iii)\ $\Longrightarrow$\ (i) of
Theorem \ref{[I_alpha,b]}, we infer that
$\|b\|_{\mathrm{PC}^\beta}\lesssim
\|[b,I_\alpha^{\gamma+}]\|_{L^r\to L^q}$ and hence (i) holds.
Using \eqref{20250119.1120} and \cite[Theorem 3.1]{km(am-2024)},
we obtain (xi)\ $\Longrightarrow$\ (iii) and hence (xi)\
$\Longrightarrow$\ (i). From the same reason, it follows that
(xiii) $\Longrightarrow$ (v) and hence (xiii) $\Longrightarrow$
(i). In conclusion, we have completed the proof the mutual
equivalences of (i) and (viii) through (xiii) of Theorem
\ref{commutator theorem 2} and hence Theorem
\ref{commutator theorem 2}.
\end{proof}

\bigskip

\noindent Mingming Cao

\medskip

\noindent Instituto de Ciencias Matem\'aticas
CSIC-UAM-UC3M-UCM, Consejo Superior de
Investigacion-\\es Cient\'ificas, C/Nicol\'as
Cabrera, 13-15, 28049 Madrid, Spain

\smallskip

\noindent{\it E-mail:} \texttt{mingming.cao@icmat.es}

\bigskip

\noindent Weiyi Kong, Dachun
Yang and Wen Yuan (Corresponding author)

\medskip

\noindent Laboratory of Mathematics
and Complex Systems (Ministry of Education of China),
School of Mathematical Sciences,
Beijing Normal University,
Beijing 100875, The People's Republic of China

\smallskip

\noindent{\it E-mails:} \texttt{weiyikong@mail.bnu.edu.cn} (W. Kong)

\noindent\phantom{{\it E-mails:} }\texttt{dcyang@bnu.edu.cn} (D. Yang)

\noindent\phantom{{\it E-mails:} }\texttt{wenyuan@bnu.edu.cn} (W. Yuan)

\bigskip

\noindent Chenfeng Zhu

\medskip

\noindent School of Mathematical Sciences, Zhejiang
University of Technology, Hangzhou 310023, The People's Republic of China

\smallskip

\noindent{\it E-mail:} \texttt{chenfengzhu@zjut.edu.cn}

\end{document}